\documentclass[11pt,letterpaper,reqno]{article}
\usepackage[utf8]{inputenc}
\usepackage[affil-it]{authblk}
\usepackage{amsfonts, amsmath, amssymb, amsthm, bbm, color, enumerate, graphicx, mathtools, tikz, hyperref, relsize}
\usepackage[numbers,sort&compress]{natbib}
\usepackage[sc]{mathpazo}
\usepackage[T1]{fontenc}
\usepackage{eulervm}
\usepackage[letterpaper, left = 1in, right = 1in, top = 1 in, bottom = 1in]{geometry}

\usepackage{subcaption} 

\usepackage{tikz-3dplot}
\usetikzlibrary{decorations.text, decorations.markings, mindmap,trees,shadows, decorations.pathreplacing}
\tikzstyle{doublearr}=[latex-latex,red, line width=0.5pt]
\tikzstyle{doublearr2}=[latex-latex,green!80!black, line width=0.5pt]
\usepackage{xcolor}

\newcommand{\dis}{\text{\textsc{dis}}}
\newcommand{\com}{\text{\textsc{com}}}
\newcommand{\qq}{\mathbf{q}}
\newcommand{\QQ}{\mathbf{Q}}
\newcommand{\bQ}{\bar{Q}}

\newcommand{\XX}{\mathbf{X}}

\newcommand{\GG}{\mathbf{G}}

\newcommand{\ER}{\mathrm{ER}}

\def\sss{}
\newcommand{\Op}{\mathrm{O}_{\sss P}}
\newcommand{\op}{\mathrm{o}_{\sss P}}

\newcommand{\N}{\mathbbm{N}}
\newcommand{\R}{\mathbbm{R}}

\newcommand{\cS}{\mathcal{S}}

\newcommand{\cE}{\mathcal{E}}

\newcommand{\pto}{\ensuremath{\xrightarrow{\mathbbm{P}}}}  
\newcommand{\dto}{\ensuremath{\xrightarrow{\mathcal{L}}}}  

\newcommand\Pro[1]{\mathbbm{P}\left(#1\right)}  

\newtheorem{theorem}{Theorem}
\newtheorem{corollary}[theorem]{Corollary}
\newtheorem{lemma}[theorem]{Lemma}
\newtheorem{proposition}[theorem]{Proposition}

\newtheorem{definition}{Definition}

\hypersetup{
 colorlinks,
 linkcolor=red,          
 citecolor=blue,       
 filecolor=red,   
 urlcolor=black, 
 pdftitle={},
 pdfauthor={},
 pdfcreator={Debankur Mukherjee},
 pdfsubject={},
 pdfkeywords={}
}

\newcommand{\CJSQ}{\mathrm{CJSQ}}

\let\plainqed\qedsymbol

\begin{document}

\title{Asymptotically Optimal Load Balancing Topologies}

\author[1]{Debankur Mukherjee\footnote{d.mukherjee@tue.nl;\ \ $^\dagger$s.c.borst@tue.nl;\ \  $^\ddagger$j.s.h.v.leeuwaarden@tue.nl}}
\author[1,2]{Sem C.~Borst$^\dagger$}
\author[1]{Johan S.H.~van Leeuwaarden$^\ddagger$}
\affil[1]{
Eindhoven University of Technology, The Netherlands}
\affil[2]{Nokia Bell Labs, Murray Hill, NJ, USA}

\renewcommand\Authands{, }

\maketitle

\begin{abstract}
We consider a system of $N$~servers inter-connected by some underlying graph topology~$G_N$.  Tasks with unit-mean exponential processing times arrive at the various servers as independent Poisson processes of rate~$\lambda$. Each incoming task is irrevocably assigned to whichever server has the smallest number of tasks among the one where it appears and its neighbors in~$G_N$.  

The above model arises in the context of load balancing in large-scale cloud networks and data centers, and has been extensively investigated in the case $G_N$ is a clique.  Since the servers are exchangeable in that case, mean-field limits apply, and in particular it has been proved that for any $\lambda < 1$, the fraction of servers with two or more tasks vanishes in the limit as $N \to \infty$. For an arbitrary graph $G_N$, mean-field techniques break down, complicating the analysis, and the queue length process tends to be worse than for a clique.  Accordingly, a graph $G_N$ is said to be $N$-optimal or $\sqrt{N}$-optimal when the queue length process on $G_N$ is equivalent to that on a clique on an $N$-scale or $\sqrt{N}$-scale, respectively.

We prove that if $G_N$ is an Erd\H{o}s-R\'enyi random graph with average degree $d(N)$, then with high probability it is $N$-optimal and $\sqrt{N}$-optimal if $d(N) \to \infty$ and $d(N) / (\sqrt{N} \log(N)) \to \infty$ as $N \to \infty$, respectively.  This demonstrates that optimality can be maintained at $N$-scale and $\sqrt{N}$-scale while reducing the number of connections by nearly a factor $N$ and $\sqrt{N} / \log(N)$ compared to a clique, provided the topology is suitably random. It is further shown that if $G_N$ contains $\Theta(N)$ bounded-degree nodes, then it cannot be $N$-optimal.  In addition, we establish that an arbitrary graph $G_N$ is $N$-optimal when its minimum degree is $N - o(N)$, and may not be $N$-optimal even when its minimum degree is $c N + o(N)$ for any $0 < c < 1/2$. Simulation experiments are conducted for various scenarios to corroborate the asymptotic results.
\end{abstract}

\section{Introduction}

\paragraph*{Background and motivation.}

In the present paper we explore the impact of the network topology
on the performance of load-balancing schemes in large-scale systems.
Load balancing algorithms play a key role in distributing service
requests or tasks (e.g.~compute jobs, data base look-ups, file
transfers, transactions) among servers in parallel-processing systems.
Well-designed load balancing schemes provide an effective mechanism
for improving relevant performance metrics experienced by users
while achieving high resource utilization levels.
The analysis and design of load balancing algorithms has attracted
strong renewed interest in recent years, mainly urged by huge
scalability challenges in large-scale cloud networks and data
centers with immense numbers of servers.

In order to examine the impact of the network topology,
we focus on a system of $N$~servers inter-connected by some
underlying graph~$G_N$.
Tasks with unit-mean exponential processing times arrive at the various servers as independent Poisson processes
of rate~$\lambda$.
Each incoming task is immediately assigned to whichever server has
the smallest number of tasks among the one where it arrives and its
neighbors in~$G_N$.  

The above model has been extensively investigated in case $G_N$ is a clique.
In that case, each task is assigned to the server with the smallest
number of tasks across the entire system, which is commonly referred
to as the Join-the-Shortest Queue (JSQ) policy.
Under the above Markovian assumptions, the JSQ policy has strong
stochastic optimality properties~\cite{EVW80,Winston77, towsley,Towsley95}.
Specifically, the queue length process is better balanced
and smaller in a majorization sense than under any alternative
non-anticipating task assignment strategy that does not have
advance knowledge of the service times.
By implication, the JSQ policy minimizes the mean overall queue length,
and hence the mean waiting time as well.
Since the servers are exchangeable in a clique topology,
the queue length process is in fact quite tractable via mean-field limits.
In particular, it can be shown that for any $\lambda < 1$, the stationary fraction
of servers with two or more tasks as well as the mean waiting time
vanish in the limit as $N \to \infty$.

Unfortunately, however, implementation of the JSQ policy in a clique
topology raises two fundamental scalability concerns.
First of all, for each incoming task the queue lengths need to be
checked at all servers, giving rise to a prohibitive communication
overhead in large-scale systems with massive numbers of servers.
Second, executing a task commonly involves the use of some data,
and storing such data for all possible tasks on all servers will
typically require an excessive amount of storage capacity~\cite{XYL16, WZYTZ16}.
These two burdens can be effectively mitigated in sparser graph
topologies where tasks that arrive at a specific server~$i$ are only
allowed to be forwarded to a subset of the servers ${\mathcal N}_i$.
For the tasks that arrive at server~$i$, queue length information
then only needs to be obtained from servers in ${\mathcal N}_i$,
and it suffices to store replicas of the required data on the
servers in ${\mathcal N}_i$.
The subset ${\mathcal N}_i$ containing the peers of server~$i$ can
be naturally viewed as its neighbors in some graph topology~$G_N$.
In the present paper we consider the case of undirected graphs,
but most of the analysis can be extended to directed graphs.

While sparser graph topologies relieve the scalability issues
associated with a clique, they defy classical mean-field techniques, and the queue length process will be worse
(in the majorization sense) because of the limited connectivity.
Surprisingly, however, even much sparser graphs can asymptotically
match the optimal performance of a clique, provided they are
suitably random, as we will further describe below.

\paragraph*{Related work.}

The above model has been studied in \cite{G15,T98}, focusing
on certain fixed-degree graphs and in particular ring topologies.
The results demonstrate that the flexibility to forward tasks
to a few neighbors, or even just one, with possibly shorter queues
significantly improves the performance in terms of the waiting time
and tail distribution of the queue length.
This resembles the so-called `power-of-two' effect in the classical
case of a complete graph where tasks are assigned to the shortest
queue among $d$~servers selected uniformly at random.
As shown by Mitzenmacher \cite{Mitzenmacher96,Mitzenmacher01}
and Vvedenskaya {\em et al.}~\cite{VDK96}, such a `power-of-$d$' scheme
provides a huge performance improvement over purely random assignment,
even when $d = 2$, in particular super-exponential tail decay,
translating into far better waiting-time performance.
Further related problems have been investigated
in \cite{M96focs,ACM01,KM00,MPS02}.
However, the results in \cite{G15,T98} also establish that
the performance sensitively depends on the underlying graph topology,
and that selecting from a fixed set of $d - 1$ neighbors typically
does not match the performance of re-sampling $d - 1$ alternate
servers for each incoming task from the entire population,
as in the power-of-$d$ scheme in a complete graph.
In contrast, when the number of neighbors $d(N)$ grows with the total number of servers~$N$, our results indicate that the performance impact of the graph topology diminishes, and that, remarkably, a broad class of suitably random topologies match the asymptotically optimal performance that is achieved in a clique or when $d(N)$ alternate servers are resampled for each incoming task~\cite{MBLW16-3}.

If tasks do not get served and never depart but simply accumulate,
then our model as described above amounts to a so-called
balls-and-bins problem on a graph.
Viewed from that angle, a close counterpart of our problem is
studied in Kenthapadi and Panigrahy~\cite{KP06}, where in our
terminology each arriving task is routed to the shortest
of $d \geq 2$ randomly selected neighboring queues. 
In this setup they show that if the underlying graph is almost
regular with degree $N^{\varepsilon}$, where $\varepsilon$ is not
too small, the maximum number of balls in a bin scales as
$\log(\log(N))/\log(d) + O(1)$.
This scaling is the same as in the case when the underlying graph
is a clique~\cite{ABKU94}.
In a more recent paper by Peres, Talwar, and Weider~\cite{PTW15} the balls-and-bins problem has been analyzed in the context of a $(1+\beta)$-choice process, where each ball goes to a random bin with probability $1-\beta$ and to the lesser loaded of the two bins corresponding to the nodes of a random edge of the graph with probability $\beta$.
In particular, for this process they show that the difference between the maximum number of balls in a bin and the typical number of balls in the bins  is $O(\log(N)/\sigma)$, where $\sigma$ is the edge expansion property of the underlying graph.
The classical balls-and-bins problem with a power-of-$d$ scheme
(often referred to as `multiple-choice' algorithm), without any graph
topology, has also been studied extensively~\cite{ABKU94,BCSV06}.
Just like in the queueing scenario mentioned above, the power-of-$d$
scheme provides a major improvement over purely random assignment
($d = 1$) where the maximum number of balls in a bin scales as
$\log(N)/\log(\log(N))$~\cite{G81}.
Several further variations and extensions have been considered
subsequently \cite{V99,ACMR95,BCSV00,CMS95,DM93,FPSS05,PR04,P05}, and we refer to~\cite{W17} for a recent survey.


As alluded to above, there are natural parallels between the
balls-and-bins setup and the queueing scenario as considered
in the present paper.
These commonalities are for example reflected in the fact that
the power-of-$d$ scheme yields a similar dramatic performance
improvement over purely random assignment in both settings.
However, there are also quite fundamental differences between the
balls-and-bins setup and the queueing scenario, even in a clique
topology, besides the obvious contrasts in the performance metrics.
The distinction is for example evidenced by the fact that a simple
round-robin strategy produces a perfectly balanced allocation in
a balls-and-bins setup but is far from optimal in a queueing scenario.
In particular, the stationary fraction of servers with two or more
tasks under a round-robin strategy remains positive in the limit
as $N \to \infty$, whereas it vanishes under the JSQ policy.
On a related account, since tasks get served and eventually depart
in a queueing scenario, less balanced allocations with a large
portion of vacant servers will generate fewer service completions
and result in a larger total number of tasks.
Thus different schemes yield not only various degrees of balance,
but also variations in the aggregate number of tasks in the system.
These differences arise not only in case of a clique,
but also in arbitrary graph topologies, and hence our problem
requires a fundamentally different approach than developed
in \cite{KP06} for the balls-and-bins setup. 
Moreover, \cite{KP06} considers only the scaling of the maximum
queue length, whereas we analyze a more detailed time-varying
evolution of the entire system along with its stationary behavior.

\paragraph*{Approach and key contributions.}

As mentioned above, the queue length process in a clique will be
better balanced and smaller (in a majorization sense) than
in an arbitrary graph~$G_N$.
Accordingly, a graph $G_N$ is said to be $N$-optimal or
$\sqrt{N}$-optimal when the queue length process on $G_N$ is equivalent
to that on a clique on an $N$-scale or $\sqrt{N}$-scale, respectively.
Roughly speaking, a graph is $N$-optimal if the \emph{fraction}
of nodes with $i$~tasks, for $i=0,1,\ldots$, behaves as in a clique
as $N\to\infty$.
Since the latter fraction is zero in the limit for all $i \geq 2$
in a clique in stationarity, the fraction of servers with two or
more tasks vanishes in any graph that is $N$-optimal, implying that
the mean waiting time vanishes as well.
Furthermore, recent results for the JSQ policy~\cite{EG15} imply
that in a clique of $N$~nodes in the heavy-traffic regime
the number of nodes with zero tasks and that with two tasks both
scale as $\sqrt{N}$ as $N \to \infty$.
Again loosely speaking, a graph is $\sqrt{N}$-optimal if in the
heavy-traffic regime the number of nodes with zero tasks
and that with two tasks when scaled by $\sqrt{N}$ both evolve as
in a clique as $N \to\infty$.
Formal definitions of asymptotic optimality on an $N$-scale or
$\sqrt{N}$-scale will be introduced in Section~\ref{sec:model}.

As one of the main results, we will demonstrate that, remarkably,
asymptotic optimality can be achieved in much sparser
Erd\H{o}s-R\'enyi random graphs (ERRGs).
We prove that a sequence of ERRGs indexed by the number of vertices~$N$
with $d(N) \to \infty$ as $N \to \infty$, is $N$-optimal.
We further establish that the latter growth condition for the average
degree is in fact necessary in the sense that any graph sequence that
contains $\Theta(N)$ bounded-degree vertices cannot be $N$-optimal.
This implies that a sequence of ERRGs with finite average degree
cannot be $N$-optimal.
The growth rate condition is more stringent for optimality
on $\sqrt{N}$-scale in the heavy-traffic regime.
Specifically, we prove that a sequence of ERRGs indexed by the number
of vertices $N$ with $d(N) / (\sqrt{N} \log(N)) \to \infty$
as $N \to \infty$, is $\sqrt{N}$-optimal.

The above results demonstrate that the asymptotic optimality of cliques
on an $N$-scale and $\sqrt{N}$-scale can be achieved in graphs that are far from fully connected,
where the number of connections is reduced by nearly a factor $N$
and $\sqrt{N}/\log(N)$, respectively, provided the topologies are
suitably random in the ERRG sense.
This translates into equally significant reductions in communication
overhead and storage capacity, since both are roughly proportional
to the number of connections.


While considerably sparser graphs can achieve asymptotic optimality in the
presence of randomness, the worst-case graph instance may even
in very dense regimes (high average degree) not be optimal.
In particular, we prove that any graph sequence with minimum degree
$N-o(N)$ is $N$-optimal, but that for any $0 < c < 1/2$ one can
construct graphs with minimum degree $c N + o(N)$ which are not
$N$-optimal for some $\lambda < 1$.
Loosely speaking, this happens due to an imbalance of arrival flows between two large parts of the network, as will be explained in Section~\ref{sec:det} in greater detail.

The key challenge in the analysis of load balancing on arbitrary graph
topologies is that one needs to keep track
of the evolution of the number of tasks at each vertex along with their
corresponding neighborhood relationship.
This creates a major problem in constructing a tractable Markovian
state descriptor, and renders a classical mean-field analysis of such processes elusive.
Consequently, even asymptotic results for load balancing processes
on an arbitrary graph have remained scarce so far.
To the best of our knowledge, the present paper is the first to establish the process-level as well as steady-state limits of the occupancy states have been rigorously established for a wide class of non-trivial (possibly random) topologies.
Since the mean-field techniques do not apply in the current scenario, we take a radically different approach and aim to compare the load
balancing process on an arbitrary graph with that on a clique.
Specifically, rather than analyze the behavior for a given class
of graphs or degree value, we explore for what types of topologies
and degree properties the performance is asymptotically similar
to that in a clique.

Our proof methodology builds on some recent advances in the analysis
of the power-of-$d$ algorithm where $d = d(N)$ grows with $N$
\cite{MBLW16-3,MBLW16-4}.
Specifically, we view the load balancing process on an arbitrary
graph as a `sloppy' version of that on a clique, and thus construct
several other intermediate sloppy versions.
By constructing novel couplings, we develop a method of comparing
the load balancing process on an arbitrary graph and that on a clique. 
In particular, we bound the difference between the fraction of vertices
with $i$ or more tasks in the two systems for $i = 1, 2, \dots$,
to obtain asymptotic optimality results. 
From a high level, conceptually related graph conditions for
asymptotic optimality were examined using quite different
techniques by Tsitsiklis and Xu \cite{TX17,TX13} in a dynamic
scheduling framework (as opposed to load balancing context).

\paragraph*{Organization of the paper.}

The remainder of the paper is organized as follows.
In Section~\ref{sec:model} we 
present a detailed model description
and introduce some useful notation and preliminaries.
Sufficient and necessary criteria for asymptotic optimality of deterministic graph sequences are developed in Sections~\ref{sec:det} and~\ref{sec:necessary}, respectively.
In Section~\ref{sec:random} we analyze asymptotic optimality of a sequence of random graph topologies.
In Section~\ref{sec:simulation} we present simulation experiments to support the analytical results, and examine the performance of topologies that are not analytically tractable.
We make a few brief concluding remarks and offer some suggestions
for further research in Section~\ref{sec:conclusion}.
Proofs of statements marked ($\bigstar$) have been provided in the
appendix.
We adopt the usual notations O($\cdot$), o($\cdot$), $\omega(\cdot)$, and $\Omega(\cdot)$ to describe asymptotic comparisons.
For a sequence of probability measures $(\mathbb{P}_N)_{N\geq 1}$, the sequence of events $(\mathcal{E}_N)_{N\geq 1}$ is said to hold with high probability if $\mathbb{P}_N(\mathcal{E}_N)\to 1$ as $N\to\infty$.
Also, for some positive function $f(N):\N\to\R_+$, we write a sequence of random variables $X_N$ is $\Op(f(N))$ or $\op(f(N))$ if $\{X_N/f(N)\}_{N\geq 1}$ is a tight sequence of random variables or converges to zero as $N\to\infty$, respectively.
The symbols `$\dto$' and `$\pto$' will denote  convergences in distribution and in probability, respectively.

\section{Model description and preliminaries}\label{sec:model}

Let $\{G_N\}_{N\geq 1}$ be a sequence of simple graphs indexed by the number of vertices $N$.
For the $N$-th system with $N$ servers, we assume that the servers are inter-connected by the underlying graph topology $G_N$, where server $i$ is identified with vertex $i$ in $G_N$, $i=1,2,\ldots, N$.
Tasks with unit-mean exponential processing times arrive at the various servers as independent Poisson processes of rate $\lambda$.
 Each server has its own queue with a fixed buffer capacity $b$ (possibly infinite).
 When a task appears at a server $i$, it is immediately assigned to the server with the shortest queue among server $i$ and its neighborhood in $G_N$.
If there are multiple such servers, one of them is chosen uniformly at random.
If $b<\infty$, and server $i$ and all its neighbors have $b$ tasks (including the ones in service), then the newly arrived task is discarded.
The service order at each of the queues is assumed to be oblivious to the actual service times, e.g.~First-Come-First-Served (FCFS).

For $k = 1,\ldots, N$, denote by $X_k(G_N,t)$ the queue length at the $k$-th server at time $t$ (including the one possibly in service), and by $X_{(k)}(G_N,t)$ the queue length at the $k$-th ordered server at time $t$ when the servers are arranged in  nondecreasing order of their queue lengths (ties can be broken in some way that will be evident from the context).
Let $Q_i(G_N, t)$ denote the number of servers with queue length at least $i$ at time $t$, $i = 1, 2,\ldots, b$, and $q_i(G_N,t):=Q_i(G_N,t)/N$ denote the corresponding fractions.
It is important to note that $\{(q_i(G_N, t))_{i\geq 1}\}_{t\geq 0}$ is itself \emph{not} a Markov process, but the joint process $\{(q_i(G_N, t))_{i\geq 1}, (X_{k}(G_N,t))_{k=1}^N\}_{t\geq 0}$ is Markov.
\begin{proposition}
\label{prop:ergod}
For any $\lambda<1$, the joint system occupancy process $\{(q_i(G_N, t))_{i\geq 1}, (X_{k}(G_N,t))_{k=1}^N\}_{t\geq 0}$ has a unique steady state 
$((q_i(G_N, \infty))_{i\geq 1}, (X_{k}(G_N,\infty))_{k=1}^N)$.
Also, the sequence of marginal random variables $\{(q_i(G_N, \infty))_{i\geq 1}\}_{N\geq 1}$ is tight with respect to the $\ell_1$-topology.
\end{proposition}
\begin{proof}[Proof of Proposition~\ref{prop:ergod}]
Note that if $b<\infty$, the process $\{(q_i(G_N, t))_{i\geq 1}, (X_{k}(G_N,t))_{k=1}^N\}_{t\geq 0}$ is clearly ergodic for all $N\geq 1$.
When $b=\infty$, to prove the ergodicity of the process, first fix any $N\geq 1$ and observe that 
the ergodicity of the queue length processes at the various vertices amounts to proving the ergodicity of the total number of tasks in the system.
Using the S-coupling and Proposition~\ref{prop:det ord} in Appendix~\ref{app:stoch} we obtain for all $t>0$,
\begin{equation}\label{eq:comparison}
\sum_{i=m}^\infty Q_i(G_N,t) \leq \sum_{i=m}^\infty Q_i(G_N',t),\quad\mbox{for all } m = 1,2, \ldots,
\end{equation}
provided the inequality holds at time $t=0$,
where $G_N'$ is the collection of $N$ isolated vertices.
Thus in particular, the total number of tasks in the system with $G_N$ is upper bounded by that with $G_N'$.
Now the queue length process on $G_N'$ is clearly ergodic since it is the collection of independent subcritical M/M/1 queues.
Next, for the $\ell_1$-tightness of $\{(q_i(G_N, \infty))_{i\geq 1}\}_{N\geq 1}$, we will use the following tightness criterion:
Define 
\begin{equation}\label{eq:s-space}
\mathcal{X} =
\left\{\qq \in [0, 1]^b: q_i \leq q_{i-1} \mbox{ for all } i = 2, \dots, b, \mbox{ and } \sum_{i=1}^b q_i < \infty\right\}
\end{equation}
as the set of all possible fluid-scaled occupancy states equipped with $\ell_1$-topology.
\begin{lemma}[{\cite[Lemma 4.7]{MBLW16-3}}]
\label{lem:tightcond}
Let $\big\{\XX^N\big\}_{N\geq 1}$ be a sequence of random variables in $\mathcal{X}$. Then the following are equivalent: 
\begin{enumerate}[{\normalfont (i)}]
\item $\big\{\XX^N\big\}_{N\geq 1}$ is tight with respect to product topology, and
for all $\varepsilon>0,$
\begin{equation}\label{eq:smalltail}
\lim_{k\to\infty}\varlimsup_{N\to\infty}\mathbb{P}\Big(\sum_{i\geq k}X_i^N>\varepsilon\Big) = 0.
\end{equation}
\item $\big\{\XX^N\big\}_{N\geq 1}$ is tight with respect to $\ell_1$ topology.
\end{enumerate}
\end{lemma}
Note that since $(q_i(G_N, \infty))_{i\geq 1}$ takes value in $[0,1]^\infty$, which is compact with respect to the product topology, Prohorov's theorem implies that $\big\{(q_i(G_N, \infty))_{i\geq 1}\big\}_{N\geq 1}$ is tight with respect to the product topology.
To verify the condition in~\eqref{eq:smalltail}, note that for each $m\geq 1$, Equation~\eqref{eq:comparison} yields
\begin{align*}
&\varlimsup_{N\to\infty}\mathbb{P}\Big(\sum_{i\geq m}q_i(G_N, \infty)>\varepsilon\Big)
\leq \varlimsup_{N\to\infty}\mathbb{P}\Big(\sum_{i\geq m}q_i(G_N', \infty)>\varepsilon\Big)
= (1-\lambda)\sum_{i\geq m}\lambda^i.
\end{align*}
Since $\lambda<1$, taking the limit $k\to\infty$, the right side of the above inequality tends to zero, and hence, the condition in~\eqref{eq:smalltail} is satisfied.
\end{proof}

\noindent
{\bf Asymptotic behavior of occupancy processes in cliques.}
We now describe the behavior of the occupancy processes on a clique
as the number of servers $N$ grows large.
Rigorous descriptions of the limiting processes are provided in Appendix~\ref{app:jsq}.

The behavior on $N$-scale is observed in terms of the fractions 
$q_i(G_N, t) = Q_i(G_N, t)/N$ of servers with queue length
at least~$i$ at time~$t$.
When $\lambda<1$, on any finite time interval,
\begin{equation}\label{eq:fluid-conv}
\big\{(q_1(K_N,t), q_2(K_N,t),\ldots)\big\}_{t\geq 0} \dto \big\{(q_1(t), q_2(t),\ldots)\big\}_{t\geq 0},
\end{equation}
as $N\to\infty$,
where $(q_1(\cdot), q_2(\cdot),\ldots)$ is some deterministic process.
Furthermore, in steady state 
\begin{equation}\label{eq:clique-stn}
q_1(K_N,\infty)\pto \lambda \quad \mbox{and} \quad
q_i(K_N,\infty)\pto 0 \ \mbox{ for all } i = 2, \dots, b,
\end{equation}
as  $N\to\infty$.
Note that $q_1(K_N,\cdot)$ is the fraction of non-empty servers. 
Thus $q_1(K_N,\infty)$ is the steady-state scaled departure rate
which should be equal to the scaled arrival rate $\lambda$.
Surprisingly, however, we observe that the steady-state fraction
of servers with a queue length of two or larger is asymptotically
negligible.

To analyze the behavior on $\sqrt{N}$-scale, we consider
a heavy-traffic scenario (a.k.a.~Halfin-Whitt regime) where the
arrival rate at each server is given by $\lambda(N)/N$ with
\begin{equation}\label{eq:HW}
(N - \lambda(N)) / \sqrt{N} \to \beta>0\quad\text{as}\quad N \to \infty.
\end{equation}
In order to describe the behavior in the limit, let 
\[\bar{\QQ}(G_N,t) =
\big(\bar{Q}_1(G_N,t), \bar{Q}_2(G_N,t), \dots, \bar{Q}_b(G_N,t)\big)\]
be a properly centered and scaled version of the occupancy process
$\QQ(G_N,t)$, with 
\begin{equation}\label{eq:HWOcc}
\bar{Q}_1(G_N,t) = - \frac{N-Q_1(G_N,t)}{ \sqrt{N}},\qquad \bar{Q}_i(G_N,t) =\frac{ Q_i(G_N,t)}{\sqrt{N}},
\end{equation}
$i = 2, \dots, b$.
The reason why $Q_1(\cdot,\cdot)$ is centered around~$N$ while $Q_i(\cdot,\cdot)$,
$i = 2, \dots, b$, are not, is because for $G_N=K_N$, the fraction of servers with a queue length of exactly one tends to one, whereas the fraction of servers with a queue length of two or larger tends to zero as $N\to\infty$, as mentioned above.
Recent results for $\QQ(K_N,t)$~\cite{EG15} show that from a suitable starting state, 
\begin{equation}\label{eq:diff-conv}
\begin{split}
&\big\{(\bQ_1(K_N,t), \bQ_2(K_N,t),\bQ_3(K_N,t),\bQ_4(K_N,t),\ldots)\big\}_{t\geq 0}\dto \big\{(\bQ_1(t), \bQ_2(t),0,0,\ldots)\big\}_{t\geq 0},
\end{split}
\end{equation}
 as $N\to\infty$, where $(\bQ_1(\cdot), \bQ_2(\cdot))$ is some diffusion process.
A precise description of the limiting diffusion process is provided in Theorem~\ref{diffusionjsqd} in Appendix~\ref{app:jsq}.
This implies that over any finite time interval,
there will be $O_P(\sqrt{N})$ servers with queue length zero
and $O_P(\sqrt{N})$ servers with a queue length of two or larger,
and hence all but $O_P(\sqrt{N})$ servers have a queue length of exactly one.

\vspace{.25cm}
\noindent
{\bf Asymptotic optimality.}
As stated in the introduction, a clique is an optimal load balancing topology, as the occupancy process is better balanced and smaller (in a majorization sense) than in any other graph topology.
In general the optimality is strict, but it turns out that near-optimality can be achieved asymptotically in a broad class of other graph topologies.
Therefore, we now introduce two notions of \emph{asymptotic optimality}, which will be useful to characterize the performance in large-scale systems. 

\begin{definition}[{Asymptotic optimality}]
A graph sequence $\GG = \{G_N\}_{N\geq 1}$  is called `asymptotically optimal on $N$-scale' or `$N$-optimal', if for any $\lambda<1$, on any finite time interval, the scaled occupancy process $(q_1(G_N,\cdot), q_2(G_N,\cdot),\ldots)$ converges weakly to the process $(q_1(\cdot), q_2(\cdot),\ldots)$ given by~\eqref{eq:fluid-conv}.

Moreover, a graph sequence $\GG = \{G_N\}_{N\geq 1}$  is called `asymptotically optimal on $\sqrt{N}$-scale' or `$\sqrt{N}$-optimal', if for any $\lambda(N)$ satisfying~\eqref{eq:HW}, on any finite time interval, the centered scaled occupancy process   $(\bQ_1(G_N,\cdot), \bQ_2(G_N,\cdot),\ldots)$ as in~\eqref{eq:HWOcc} converges weakly to the process $(\bQ_1(\cdot), \bQ_2(\cdot),\ldots)$ given by~\eqref{eq:diff-conv}.
\end{definition}
\noindent
Intuitively speaking, if a graph sequence is $N$-optimal or $\sqrt{N}$-optimal, then in some sense, the associated occupancy processes are indistinguishable from those of the sequence of cliques on $N$-scale or $\sqrt{N}$-scale.
In other words, on any finite time interval their occupancy processes can differ from those in cliques by at most $o(N)$ or $o(\sqrt{N})$, respectively. 
For brevity, $N$-scale and $\sqrt{N}$-scale are often referred to as \emph{fluid scale} and \emph{diffusion scale}, respectively.
In particular, due to the $\ell_1$-tightness of the scaled
occupancy processes as stated in Proposition~\ref{prop:ergod},
we obtain that for any $N$-optimal graph sequence $\{G_N\}_{N\geq 1}$,
\begin{equation}
q_1(G_N,\infty)\to \lambda \quad \mbox{and} \quad
q_i(G_N,\infty)\to 0 \ \mbox{ for all } i = 2, \dots, b,
\end{equation}
as  $N\to\infty$,
implying that the stationary fraction of servers with queue length
two or larger and the mean waiting time vanish.



\section{Sufficient criteria for asymptotic optimality}\label{sec:det}

In this section we develop a criterion for asymptotic optimality
of an arbitrary deterministic graph sequence on different scales.
In Section~\ref{sec:random} this criterion will be leveraged to establish optimality
of a sequence of random graphs.

We start by introducing some useful notation, and two measures
of \emph{well-connectedness}.
Let $G=(V,E)$ be any graph.
For a subset $U\subseteq V$, define $\com(U) := |V\setminus N[U]|$
to be the set of all vertices that are disjoint from $U$,
where $N[U]:=U\cup \{v\in V:\ \exists\ u\in U\mbox{ with }(u,v)\in E\}$.
For any fixed $\varepsilon>0$ define
\begin{equation}\label{def:dis}
\begin{split}
\dis_1(G,\varepsilon) &:= \sup_{U\subseteq V, |U|\geq \varepsilon |V|}\com(U),\\
\dis_2(G,\varepsilon) &:= \sup_{U\subseteq V, |U|\geq \varepsilon \sqrt{|V|}}\com(U).
\end{split}
\end{equation}

The next theorem provides sufficient conditions for asymptotic
optimality on $N$-scale and $\sqrt{N}$-scale in terms of the above
two well-connectedness measures.

\begin{theorem}\label{th:det-seq}
For any graph sequence $\GG= \{G_N\}_{N\geq 1}$,
\begin{enumerate}[{\normalfont (i)}]
\item $\GG$ is $N$-optimal if for any $\varepsilon>0$, 
$\dis_1(G_N,\varepsilon)/N\to 0$,  as $ N\to\infty.$
\item $\GG$ is $\sqrt{N}$-optimal if for any $\varepsilon>0$, 
$\dis_2(G_N,\varepsilon)/\sqrt{N}\to 0$,  as $ N\to\infty.$
\end{enumerate}
\end{theorem}

The next corollary is an immediate consequence
of Theorem~\ref{th:det-seq}.

\begin{corollary}
Let $\GG= \{G_N\}_{N\geq 1}$ be any graph sequence and $d_{\min}(G_N)$ be the minimum degree of $G_N$. Then
{\rm(i)} If $d_{\min}(G_N) = N-o(N)$, then $\GG$ is $N$-optimal, and 
{\rm(ii)} If $d_{\min}(G_N) = N-o(\sqrt{N})$, then $\GG$ is $\sqrt{N}$-optimal.
\end{corollary}

The rest of the section is devoted to a discussion of the main
proof arguments for Theorem~\ref{th:det-seq}, focusing on the
proof of $N$-optimality. 
The proof of $\sqrt{N}$-optimality follows along similar lines.
We establish in Proposition~\ref{prop:cjsq} that if a system is able
to assign each task to a server in the set $\cS^N(n(N))$ of the
$n(N)$ nodes with shortest queues (ties broken arbitrarily), where $n(N)$ is $o(N)$,
then it is $N$-optimal. 
Since the underlying graph is not a clique however (otherwise there
is nothing to prove), for any $n(N)$ not every arriving task can be
assigned to a server in $\cS^N(n(N))$.
Hence we further prove in Proposition~\ref{prop:stoch-ord-new}
a stochastic comparison property implying that if on any finite
time interval of length~$t$, the number of tasks $\Delta^N(t)$
that are not assigned to a server in $\cS^N(n(N))$ is $o_P(N)$,
then the system is $N$-optimal as well.
The $N$-optimality can then be concluded when $\Delta^N(t)$ is
$o_P(N)$, which we establish in Proposition~\ref{prop:dis-new}
under the condition that $\dis_1(G_N,\varepsilon)/N\to 0$ as
$N \to \infty$ as stated in Theorem~\ref{th:det-seq}.

To further explain the idea described in the above proof outline,
it is useful to adopt a slightly different point of view towards
load balancing processes on graphs.
From a high level, a load balancing process can be thought of as follows:
there are $N$ servers, which are assigned incoming tasks by some scheme.
The assignment scheme can arise from some topological structure as considered in this paper, in which case we will call it \emph{topological load balancing}, or it can arise from some other property of the occupancy process, in which case we will call it \emph{non-topological load balancing}.
As mentioned earlier, under Markovian assumptions, the JSQ policy or the clique is optimal among the set of all non-anticipating schemes, irrespective of being topological or non-topological.
Also, load balancing on graph topologies other than a clique can be thought of as a `sloppy' version of that on a clique, when each server only has access to partial information on the occupancy state.
Below we first introduce a different type of sloppiness in the task assignment scheme, and show that under a limited amount of sloppiness optimality is retained on a suitable scale.
Next we will construct a scheme which is a hybrid of topological and non-topological schemes, whose behavior is simultaneously close to both the load balancing process on a suitable graph and that on a clique.

\vspace{.25cm}
\noindent
{\bf A class of sloppy load balancing schemes.}
Fix some function $n:\N\to\N$, and recall the set $\cS^N(n(N))$ as before.
Consider the class $\CJSQ(n(N))$ where each arriving task is assigned
to one of the servers in $\cS^N(n(N))$. 
It should be emphasized that for any scheme in $\CJSQ(n(N))$, we are not imposing any restrictions on how the ties are broken to select the specific set $\cS^N(n(N))$, or how the incoming task should be assigned to a server in $\cS^N(n(N))$.
The scheme only needs to ensure that the arriving task is assigned to \emph{some} server in $\cS^N(n(N))$ with respect to \emph{some} tie breaking mechanism.
The next proposition provides a sufficient criterion for asymptotic
optimality of any scheme in $\CJSQ(n(N))$.

\begin{proposition}[$\bigstar$]\label{prop:cjsq}
For $0\leq n(N)<N$, let $\Pi\in\CJSQ(n(N))$ be any scheme. {\rm(i)} If $n(N)/N\to 0$ as $N\to\infty,$ then $\Pi$ is $N$-optimal, and {\rm(ii)} If $n(N)/\sqrt{N}\to 0$ as $N\to\infty,$ then $\Pi$ is $\sqrt{N}$-optimal.
\end{proposition}

\vspace{.25cm}
\noindent
{\bf A bridge between topological and non-topological load balancing.}
For any graph $G_N$ and $n \leq N$, we first construct a scheme called $I(G_N, n)$, which is an intermediate blend between the topological load balancing process on $G_N$ and some kind of non-topological load balancing on $N$ servers.
The choice of $n = n(N)$ will be clear from the context.

To describe the scheme $I(G_N,n)$, first synchronize the arrival
epochs at server~$v$ in both systems, $v = 1,2,\ldots,N$.
Further, the servers in both systems are arranged in  non-decreasing order of the queue lengths, and  the departure epochs at the $k$-th ordered
server in the two systems are synchronized, 
$k = 1,2,\ldots,N$.
When a task arrives at server~$v$ at time~$t$ say, it is assigned
in the graph $G_N$ to a server $v' \in N[v]$ according to its own
statistical law.
For the assignment under the scheme $I(G_N,n)$, first observe that if
\begin{equation}\label{eq:criteria}
\min_{u\in N[v]}X_u(G_N,t) \leq \max_{u\in \cS(n)}X_u(G_N,t),
\end{equation}
then there exists \emph{some} tie-breaking mechanism for which
$v' \in N[v]$ belongs to $\cS(n)$ under $G_N$.
Pick such an ordering of the servers, and assume that $v'$ is the
$k$-th ordered server in that ordering, for some $k \leq n+1$.
Under $I(G_N,n)$ assign the arriving task to the $k$-th ordered
server (breaking ties arbitrarily in this case).
Otherwise, if \eqref{eq:criteria} does not hold, then the task is
assigned to one of the $n+1$ servers with minimum queue lengths
under $G_N$ uniformly at random.

Denote by $\Delta^N(I(G_N,n),T)$ the cumulative number of arriving tasks
up to time $T \geq 0$ for which Equation~\eqref{eq:criteria} is
violated under the above coupling.
The next proposition shows that the load balancing process under
the scheme $I(G_N,n)$ is close to that on the graph $G_N$ in terms
of the random variable $\Delta^N(I(G_N,n),T)$.

\begin{proposition}[$\bigstar$]\label{prop:stoch-ord-new}
The following inequality is preserved almost surely
\begin{equation}\label{eq:stoch-ord-new}
\sum_{i=1}^b |Q_i(G_N,t)-Q_i(I(G_N, n),t)|\leq 2\Delta^N(I(G_N,n),t)\ \forall\ t\geq 0,
\end{equation}
provided the two systems start from the same occupancy state at $t=0$.
\end{proposition}

In order to conclude optimality on $N$-scale or $\sqrt{N}$-scale,
it remains to be shown that for any $T>0$, $\Delta^N(I(G_N,n),T)$ is sufficiently small.
The next proposition provides suitable asymptotic bounds for
$\Delta^N(I(G_N,n),T)$ under the conditions on $\dis_1(G_N,\varepsilon)$
and $\dis_2(G_N,\varepsilon)$ stated in Theorem~\ref{th:det-seq}.

\begin{proposition}\label{prop:dis-new}
For any $\varepsilon, T>0$ the following holds.
\begin{enumerate}[{\normalfont (i)}]
\item There exist $\varepsilon'>0$ and $n_{\varepsilon'}(N)$ with $n_{\varepsilon'}(N)/N\to 0$ as $N\to\infty$, such that if $\dis_1(G_N,\varepsilon')/N\to 0$ as $N\to\infty$, then  
\[\Pro{\Delta^N(I(G_N,n_{\varepsilon'}),T)/N>\varepsilon}\to 0.\] 
\item There exist $\varepsilon'>0$ and $m_{\varepsilon'}(N)$ with $m_{\varepsilon'}(N)/\sqrt{N}\to 0$ as $N\to\infty$, such that if $\dis_2(G_N,\varepsilon')/\sqrt{N}\to 0$ as $N\to\infty$, then  
\[\Pro{\Delta^N(I(G_N,m_{\varepsilon'}),T)/\sqrt{N}>\varepsilon}\to 0.\]
\end{enumerate}
\end{proposition}

The proof of Theorem~\ref{th:det-seq} then readily follows
by combining Propositions~\ref{prop:cjsq}-\ref{prop:dis-new}
and observing that the scheme $I(G_N,n)$ belongs to the class
$\CJSQ(n)$ by construction.

\begin{proof}[Proof of Proposition~\ref{prop:dis-new}]
Fix any $\varepsilon, T>0$ and choose
$\varepsilon' = \varepsilon/(2 \lambda T)$.
With the coupling described above, when a task arrives at some
vertex~$v$ say, Equation~\eqref{eq:criteria} is violated only if none
of the vertices in $\cS(n_{\varepsilon'}(N))$ is a neighbor of~$v$.
Thus, the total instantaneous rate at which this happens is
\[
\lambda \com(\cS(n_{\varepsilon'}(N),t)) \leq
\lambda \sup_{U \subseteq V_N, |U| \geq n_{\varepsilon'}(N)} \com(U),
\]
irrespective of what this set $\cS^N(n(N))$ actually is.
Therefore, for any fixed $T \geq 0$,
\[
\Delta^N(I(G_N,n_{\varepsilon'}),T) \leq
A\Big(\lambda \sup_{U \subseteq V_N, |U| \geq n_{\varepsilon'}(N)} \com(U)\Big),
\]
where $A(\cdot)$ represents a unit-rate Poisson process.
This can then be leveraged to show that
$\Delta^N(I(G_N,n_{\varepsilon'}),T)$ is small on an $N$-scale
and $\sqrt{N}$-scale, respectively, under the conditions stated
in the proposition, by choosing a suitable $n_{\varepsilon'}$.

Specifically, if $\dis_1(G_N,\varepsilon')/N \to 0$, then there exists
$n_{\varepsilon'}(N)$ with $n_{\varepsilon'}(N)/N \to 0$ such that 
$\dis_1(G_N,\varepsilon') \leq n_{\varepsilon'}(N)$ for all $N \geq 1$,
and hence
$\sup_{U \subseteq V_N, |U| \geq n_{\varepsilon'}(N)} \com(U) \leq
\varepsilon' N$.
It then follows that with high probability,
\[
\limsup_{N \to\infty} \frac{1}{N} \Delta^N(I(G_N,n_{\varepsilon'}),T) \leq
\limsup_{N \to\infty} \frac{1}{N} A\Big(\lambda T \varepsilon' N\Big) \leq
2 \lambda T \varepsilon' = \varepsilon.
\]

Likewise, if $\dis_2(G_N,\varepsilon')/\sqrt{N} \to 0$, then there exists
$m_{\varepsilon'}(N)$ with $m_{\varepsilon'}(N)/\sqrt{N}\to 0$ such that
$\dis_2(G_N,\varepsilon') \leq m_{\varepsilon'}(N)$ for all $N \geq 1$,
and hence
$\sup_{U \subseteq V_N, |U| \geq m_{\varepsilon'}(N)} \com(U) \leq
\varepsilon' \sqrt{N}$.
It then follows that with high probability,
\begin{align*}
&\limsup_{N\to\infty} \frac{1}{\sqrt{N}} \Delta^N(I(G_N,m_{\varepsilon'}),T) 
\leq
\limsup_{N\to\infty} \frac{1}{\sqrt{N}} A\Big(\lambda T \varepsilon' \sqrt{N}\Big) \leq
2 \lambda T \varepsilon' = \varepsilon.
\end{align*}
\end{proof}

\begin{proof}[Proof of Theorem~\ref{th:det-seq}]

(i) In order to prove the fluid-level optimality of $G_N$, fix any $\varepsilon>0$. 
Observe from Proposition~\ref{prop:stoch-ord-new} and Proposition~\ref{prop:dis-new} (i) that there exists $\varepsilon'>0$ such that with high probability
\begin{align*}
 \sup_{t\in[0,T]}\frac{1}{N}\sum_{i=1}^b|Q_i(G_N,t)-Q_i(I(G_N,n_{\varepsilon'}(N)),t)| 
\leq \frac{2\Delta^N_\varepsilon(T)}{N}\leq \varepsilon.
\end{align*}
Furthermore, since $I(G_N,n_{\varepsilon'}(N))\in\CJSQ(n_{\varepsilon'}(N))$ and $n_{\varepsilon'}(N)/N\to 0$, Proposition~\ref{prop:cjsq} yields
$$\sup_{t\in[0,T]}\sum_{i=1}^b|q_i(I(G_N,n_{\varepsilon'}(N)),t)-q_i(t)|\pto 0\quad \mathrm{as}\quad N\to\infty.$$
Thus since $\varepsilon>0$ is arbitrary, we obtain that with high probability as $N\to\infty$,
$$\sup_{t\in[0,T]}\sum_{i=1}^b|q_i(G_N,t)-q_i(t)|\leq \varepsilon'',$$
for all $\varepsilon''>0$, which completes the proof of Part (i).\\

(ii) To prove the diffusion-level optimality of $G_N$, again fix any $\varepsilon>0$. 
As in Part (i), using Proposition~\ref{prop:stoch-ord-new} and Proposition~\ref{prop:dis-new} (ii), there exists $\varepsilon'>0$
\begin{align*}
\sup_{t\in[0,T]}\frac{1}{\sqrt{N}}\sum_{i=1}^b|Q_i(G_N,t)-Q_i(I(G_N,m_{\varepsilon'}(N)),t)|\leq \frac{\Delta^N_{\varepsilon'}(T)}{\sqrt{N}}\leq \varepsilon.
\end{align*}
Furthermore, since $I(G_N,m_{\varepsilon'}(N))\in\CJSQ(m_{\varepsilon'}(N))$ and $m_{\varepsilon'}(N)/\sqrt{N}\to 0$, Proposition~\ref{prop:cjsq} yields
\begin{align*}
&\big\{(\bQ_1(I(G_N,m_{\varepsilon'}(N)),t), \bQ_2(I(G_N,m_{\varepsilon'}(N)),t),\ldots)\big\}_{t\geq 0}\\
&\hspace{3cm}\dto \big\{(\bQ_1(t), \bQ_2(t),\ldots)\big\}_{t\geq 0},
\end{align*}
as $N\to\infty$, where the process $(\bQ_1(\cdot), \bQ_2(\cdot),\ldots)$ given by~\eqref{eq:diff-conv}.
Since $\varepsilon>0$ is arbitrary, we thus obtain
$$\big\{(\bQ_1(G_N,t), \bQ_2(G_N,t),\ldots)\big\}_{t\geq 0} \dto \big\{(\bQ_1(t), \bQ_2(t),\ldots)\big\}_{t\geq 0},$$
as $N\to\infty$, which completes the proof of Part (ii).
\end{proof}

\section{Necessary criteria for asymptotic optimality}
\label{sec:necessary}
From the conditions of Theorem~\ref{th:det-seq} it follows that if for all $\varepsilon>0$, $\dis_1(G_N,\varepsilon)$ and $\dis_2(G_N,\varepsilon)$ are $o(N)$ and $o(\sqrt{N})$, respectively, then
the total number of edges in $G_N$ must be $\omega(N)$ and $\omega(N\sqrt{N})$, respectively.
Theorem~\ref{th:bdd-deg} below states that the \emph{super-linear} growth rate of the total number of edges is not only sufficient, but also necessary in the sense that any graph with $O(N)$ edges is asymptotically sub-optimal on $N$-scale.

\begin{theorem}
\label{th:bdd-deg}
Let $\GG= \{G_N\}_{N\geq 1}$ be any graph sequence, such that there exists a fixed integer $M<\infty$ with 
\begin{equation}\label{eq:bdd-deg}
\limsup_{N\to\infty}\dfrac{\#\big\{v\in V_N:d_v\leq M\big\}}{N}>0,
\end{equation}
where $d_v$ is the degree of the vertex $v$.
Then $\GG$ is sub-optimal on $N$-scale.
\end{theorem}

\begin{proof}[{Proof of Theorem~\ref{th:bdd-deg}}]
For brevity, denote by $\Xi_N(M)\subseteq V_N$ the set of all vertices with degree at most $M$.
Since $|\Xi_N(M)|/N\leq 1,$ from~\eqref{eq:bdd-deg} we have a convergent subsequence $\big\{\Xi_{N_n}(M)\big\}_{n\geq 1}$ with $\{N_n\}_{n\geq 1}\subseteq \N$, such that $|\Xi_{N_n}(M)|/N\to \xi>0$, as $N\to\infty$.
For the rest of the proof we will consider the asymptotic statements along this subsequence, and hence omit the subscript $n$.

Let the system start from an occupancy state where all the vertices in $\Xi_N(M)$ are empty.
We will show that in finite time, a positive fraction of vertices in $\Xi_N(M)$ will have at least two tasks. 
This will prove the fluid limit sample path cannot agree with that of the sequence of cliques, and hence $\{G_N\}_{N\geq 1}$ cannot be $N$-optimal.
The idea of the proof is as follows: If a graph contains $\Theta(N)$ bounded degree vertices, then starting from all empty servers, in any finite time interval there will be $\Theta(N)$ servers $u$ say, for which all the servers in $N[u]$ have at least one task.
For all such servers an arrival at $u$ must produce a server of queue length two.
Thus, it shows that the instantaneous rate at which servers of queue length two are formed is bounded away from zero, and hence $\Theta(N)$ servers of queue length two are produced in finite time. 

Let $u$ be a vertex with degree $M$ or less in $G_N$.
Consider the event $\cE_N(u,t)$ that at time~$t$ all vertices in $N[u]$ have at least one job.
Note that since $M<\infty$ is fixed, for any $t>0$, $\Pro{\cE_N(u,t)}\geq\delta(t)$ for some $\delta(t)>0$, for all $N\geq 1$.
To see this, note that $\delta(t)$ is the probability that before time $t$ there are $M+1$ arrivals at vertex $u$ and no departure has taken place.
Also observe that for two vertices $u,v\in V_N$ with degrees at most $M$, 
\begin{equation}
\Pro{\cE_N(u,t)\cap \cE_N(v,t)}\geq \delta(t)^2.
\end{equation}
Indeed the probability of the event $\cE_N(u,t)\cap \cE_N(v,t)$ can be lower bounded by the probability of the event that before time $t$ there are $M+1$ arrivals at vertex $u$, $M+1$ arrivals at vertex $v$, and no departure has taken place from $N[u]\cup N[v]$.
Thus, at time $t$, the fraction of vertices in $\Xi_N(M)$ for which all the neighboring vertices have at least one task, is lower bounded by $\delta(t)$.
Now the proof is completed by considering the following: let $u$ be a vertex of degree $M<\infty$ for which all the neighbors have at least one task.
Then at such an instance if a task arrives at server $u$, it must be assigned to a server with queue length one, and hence a server with queue length two will be formed.
Therefore the total scaled instantaneous rate at which the number of queue length two is being formed at time $t$ is at least $\lambda \delta(t)>0$, which also gives the total rate of increase of the fraction of vertices with at least two tasks. 
\end{proof}

\vspace{.25cm}
\noindent
{\bf Worst-case scenario.}
Next we consider the worst-case scenario.
Theorem~\ref{th:min-deg-negative} below asserts that a graph sequence can be sub-optimal for some $\lambda<1$ even when the minimum degree  $d_{\min}(G_N)$ is~$\Theta(N)$.

\begin{theorem}
\label{th:min-deg-negative}
For any $\big\{d(N)\big\}_{N\geq 1}$, such that $d(N)/N\to c$ with $0<c<1/2$, there exists $\lambda<1$, and a graph sequence $\GG= \{G_N\}_{N\geq 1}$ with $d_{\min}(G_N)= d(N)$, such that $\GG$ is sub-optimal on $N$-scale.
\end{theorem}

To construct such a sub-optimal graph sequence, consider a sequence of complete bipartite graphs $G_N = (V_N, E_N)$, with $V_N = A_N \sqcup B_N$ and $|A_N|/N\to c\in (0,1/2)$ as $N\to\infty$.
If this sequence were $N$-optimal, then starting from all empty servers, asymptotically 
the fraction of servers with queue length one would converge to $\lambda$, and the fraction of servers with queue length two or larger should remain zero throughout.
Now note that for large $N$ the rate at which tasks join the empty servers in $A_N$ is given by $(1-c)\lambda$, whereas the rate of empty server generation in $A_N$ is at most $c$.
Choosing $\lambda>c/(1-c)$, one can see that in finite time each server in $A_N$ will have at least one task.
From that time onward with at least instantaneous rate $\lambda(\lambda - c) - c$, servers with queue length two start forming.
The range for $c$ stated in Theorem~\ref{th:min-deg-negative} is only to ensure that there exists $\lambda<1$ with $\lambda(\lambda - c) - c>0$. 
\begin{proof}[{Proof sketch of Theorem~\ref{th:min-deg-negative}}]
Fix a $c>0$.
Construct the graph sequence $\big\{G_N\big\}_{N\geq 1}$ as a sequence of complete bipartite graphs with size of one partite set of the $N$-th graph to be $\lceil c N\rceil$, i.e., $V_N = A_N \sqcup B_N$, such that $|A_N| = \lceil c N\rceil$ and $B_N = V_N\setminus A_N$, and the edge set is given by $E_N = \big\{(u,v): u\in A_N, v\in B_N\big\}$.
Note that $d_{\min}(G_N)/N\to c$, as $N\to\infty$.
We will show that for any $0<c<1/2$, there exists $\lambda$, such that $\GG$ is sub-optimal on $N$-scale.

Assume on the contrary that $\GG$ is $N$-optimal.
Denote by $Q_{i,A}^N(t)$ and $Q_{i,B}^N(t)$ the number of vertices with at least $i$ tasks in partite sets $A_N$ and $B_N$, respectively.
Also define $q_{i,A}^N(t) = Q_{i,A}^N(t)/N$ and $q_{i,B}^N(t) = Q_{i,B}^N(t)/N$.
Assume $q_{2,A}^N(0)=0$, for all $N$.
Observe that as long as $c - q_{1,A}^N>0$ by a non-vanishing margin, any external arrival to servers in $B_N$ will be assigned to an empty server in $A_N$ with probability $1-\text{O}(1/N)$.
Similarly, as long as $1-c - q_{1,B}^N>0$ by a non-vanishing margin, any external arrival to servers in $A_N$ will be assigned to an empty server in $B_N$ with probability $1-\text{O}(1/N)$.
Thus one can show that as $N\to\infty$, until $q_{1,A}^N$ hits $c$, the processes $\big\{q_{1,A}^N(t)\big\}$ and $\big\{q_{2,B}^N(t)\big\}$ converges weakly to a deterministic process described by the following set of ODE's:
\begin{equation}\label{eq:fluid-counter}
\begin{split}
q_{1,A}'(t) &= \lambda (1 - c) - q_{1,A}(t)\\
q_{1,B}'(t) &= \lambda c - q_{1,B}(t).
\end{split}
\end{equation}
Since the total scaled arrival rate into the system of $N$ servers is $\lambda$, should the above system follow the fluid-limit trajectory of the occupancy process for a clique, starting from an all-empty state, $q_{1,A}(t) + q_{1,B}(t)$ must approach $\lambda$ as $t\to\infty$, and $q_{i, A}(t)$ and $q_{i,B}(t)$ both remain 0 for all $t\geq 0$, $i\geq 2$.
When $\lambda>c/(1-c)$,~\eqref{eq:fluid-counter} implies that in finite time $q_{1,A}(t)$ hits $c$.
Consequently, $q_{1,B}(t)$ should approach $\lambda - c$ as $t\to \infty$.
Now we claim that when $q_{1,A}(t) =c$, if a task appears at a server $v$ in $B_N$ that has queue length one, then with probability $1-\text{O}(1/N)$, it will be assigned to a server in $A_N$.
To see this, note that at such an arrival if there is an empty server in $A_N$, then the arriving task is clearly assigned to the idle server, otherwise, when there is no empty server in $A_N$, the arriving task is assigned uniformly at random among the vertices in $N[v]$ having queue length one.
Since there are $\Theta(N)$ vertices in $A_N$ with queue length one, the arriving task with probability $1-O(1/N)$ joins a server in $A_N$.
Therefore, the total scaled rate of tasks arriving at the servers in $A_N$ is at least $\lambda(\lambda-c)$, whereas the total scaled rate at which tasks can leave from servers in $A_N$ is at most $c$.
Thus if $\lambda(\lambda-c)>c$, then in finite time, a positive fraction of servers in $A_N$ will have queue length two or larger.
Now observe that
$$\lambda(\lambda-c)>c\implies \lambda> \frac{c+\sqrt{c^2+4c}}{2},$$
and $(c+\sqrt{c^2+4c})/2<1$ for any $c\in (0,1/2)$. This completes the proof of Theorem~\ref{th:min-deg-negative}.
\end{proof}

\section{Asymptotically optimal random graph topologies}
\label{sec:random}
In this section we use Theorem~\ref{th:det-seq} to investigate how the load balancing process behaves on random graph topologies. 
Specifically, we aim to understand what types of graphs are asymptotically optimal in the presence of randomness (i.e., in the average case scenario).
Theorem~\ref{th:inhom} below establishes sufficient conditions for asymptotic optimality of a sequence of inhomogeneous random graphs.
Recall that a graph $G' = (V',E')$ is called a supergraph of $G=(V,E)$ if $V=V'$ and $E\subseteq E'$.

\begin{theorem}
\label{th:inhom}
Let $\GG= \{G_N\}_{N\geq 1}$ be a graph sequence such that for each $N$, $G_N = (V_N, E_N)$ is a supergraph of the inhomogeneous random graph $G_N'$ where any two vertices $u, v\in V_N$ share an edge with probability $p_{uv}^N$.
\begin{enumerate}[{\normalfont (i)}]
\item If $\inf\ \{p^N_{uv}: u, v\in V_N\}$ is $\omega(1/N)$, then $\GG$ is $N$-optimal.
\item If $\inf\ \{p^N_{uv}: u, v\in V_N\}$ is $\omega(\log(N)/\sqrt{N})$, then $\GG$ is $\sqrt{N}$-optimal.
\end{enumerate}
\end{theorem}

The proof of Theorem~\ref{th:inhom} relies on Theorem~\ref{th:det-seq}.
Specifically, if $G_N$ satisfies conditions~(i) and~(ii) in
Theorem~\ref{th:inhom}, then the corresponding conditions~(i) and~(ii)
in Theorem~\ref{th:det-seq} hold.
\begin{proof}[{Proof of Theorem~\ref{th:inhom}}]
In this proof we will verify the conditions stated in Theorem~\ref{th:det-seq} for fluid and diffusion level optimality.
Fix any $\varepsilon>0$.

(i) Observe that for $G_N = (V_N, E_N)$ as described in Theorem~\ref{th:inhom} (i), we have $p(N):=\inf\ \{p^N_{uv}: u,v\in V_N\}$ with $Np(N)\to \infty$ as $N\to\infty$. 
For any two subsets $V_1$, $V_2\subseteq V_N$, denote by $E_N(V_1, V_2)$ the number of cross-edges between $V_1$ and $V_2$.
Now, for any function $n:\N\to\N$, 
\begin{equation}\label{eq:fluid-bound}
\begin{split}
&\Pro{\exists\ V_1, V_2\subseteq V_N:~|V_1|\geq\varepsilon N,\ |V_2|\geq n(N), E_N(V_1,V_2) = 0}\\&\\
&= \Pro{\exists\ V_1, V_2\subseteq V_N:|V_1|=\varepsilon N,\ |V_2|= n(N), E_N(V_1,V_2) = 0}\\&\\
&\leq {N(1-\varepsilon) \choose \varepsilon N} {N-2\varepsilon N \choose n(N)}(1-p(N))^{\varepsilon N n(N)}\\&\\
&\lesssim \frac{1}{[\varepsilon^\varepsilon(1-\varepsilon)^{1-\varepsilon}]^N}\times\frac{\big(\frac{N}{n(N)}\big)^{n(N)}}{\big(1-\frac{n(N)}{N(1-\varepsilon)}\big)^{N(1-\varepsilon)}}\times \exp(-\varepsilon Np(N) n(N))\\&\\
&\lesssim \dfrac{\exp(-\varepsilon Np(N)n(N))\times \exp(n(N)\log(N))}{\exp(N\log [\varepsilon^\varepsilon(1-\varepsilon)^{1-\varepsilon}])\exp(-n(N))},
\end{split}
\end{equation}
where the first equality is due to the fact that if there are two  sets of vertices $V_1$ and $V_2$ with  $|V_1|\geq\varepsilon N$ and $|V_2|\geq n(N)$, such that there is no edge between $V_1$ and $V_2$, then the graph must contain two sets $V_1'$ and $V_2'$ of sizes exactly equal to $\varepsilon N$ and $n(N)$, respectively, such that there is no edge between $V_1'$ and $V_2'$, and vice-versa.
Choosing $n(N) = N/\sqrt{Np(N)}$ say, it can be seen that for any $p(N)$ such that $Np(N)\to\infty$ as $N\to\infty$, $n(N)/N\to 0$ and the above probability goes to 0. Therefore for any $\varepsilon, \delta>0$, \eqref{eq:fluid-bound} yields
\begin{align*} 
&\Pro{\dis_1(G_N,\varepsilon)>\delta N}\leq \Pro{\exists\ U\subseteq V_N:\ |U|\geq \varepsilon N \mbox{ and }\com(U)\geq \delta N}\to 0,
\end{align*}
as $N\to\infty$.

(ii) Again, for $G_N = (V_N, E_N)$ as described in Theorem~\ref{th:inhom} (i), we have $p(N):=\inf\ \{p^N_{uv}: u,v\in V_N\}$ with $Np(N)/(\sqrt{N}\log(N))\to \infty$ as $N\to\infty$.  
Now as in Part (i), for any function $n:\N\to\N$,
\begin{equation}\label{eq:diff-bound}
\begin{split}
&\Pro{\exists\ V_1, V_2\subseteq V_N:|V_1|\geq\varepsilon \sqrt{N},\ |V_2|\geq n(N), E_N(V_1,V_2) = 0}
\\&\\
&= \Pro{\exists\ V_1, V_2\subseteq V_N:|V_1|=\varepsilon \sqrt{N},\ |V_2|= n(N), E_N(V_1,V_2) = 0}
\\&\\
&\leq {N-\varepsilon\sqrt{N} \choose \varepsilon \sqrt{N}} {N-2\varepsilon\sqrt{N} \choose n(N)}(1-p(N))^{\varepsilon\sqrt{N} n(N)}
\\&\\
&\lesssim N^{\varepsilon\sqrt{N}/2}\exp(\varepsilon\sqrt{N}) \times
N^{n(N)}\times \exp(-\varepsilon\sqrt{N}p(N)n(N))\\
&\hspace{1.5cm} \times\exp\Big(\frac{-\varepsilon n(N)}{\sqrt{N}}+n(N)\Big(1-\frac{n(N)}{N-\varepsilon\sqrt{N}}\Big)\Big).
\end{split}
\end{equation}
Choosing $n(N) = \sqrt{N}/\sqrt{\sqrt{N}p(N)/\log(N)}$, it can be seen that as $N\to\infty,$ $n(N)/\sqrt{N}\to 0$ and the above probability converges to 0.
 Therefore for any $\varepsilon, \delta>0$, \eqref{eq:diff-bound} yields
\begin{align*} 
&\Pro{\dis_2(G_N,\varepsilon)>\delta \sqrt{N}}
\leq \Pro{\exists\ U\subseteq V_N:\ |U|\geq \varepsilon \sqrt{N} \mbox{ and }\com(U)\geq \delta \sqrt{N}}\to 0, 
\end{align*}
as $N\to\infty$.
This completes the proof of Theorem~\ref{th:inhom}.
\end{proof}

As an immediate corollary to Theorem~\ref{th:inhom} we obtain an optimality result for the sequence of Erd\H{o}s-R\'enyi random graphs.

\begin{corollary}\label{cor:errg}
Let $\GG= \{G_N\}_{N\geq 1}$ be a graph sequence such that for each $N$,
$G_N$ is a supergraph of $\ER_N(p(N))$, and $d(N) = (N-1)p(N)$. Then
{\normalfont (i)}
If $d(N)\to\infty$ as $N\to\infty$, then $\GG$ is $N$-optimal.
{\normalfont (ii)}
If $d(N)/(\sqrt{N}\log(N))\to\infty$ as $N\to\infty$, then $\GG$ is $\sqrt{N}$-optimal.
\end{corollary}

Theorem~\ref{th:det-seq} can be further leveraged to establish the
optimality of the following sequence of random graphs.
For any $N \geq 1$ and $d(N) \leq N-1$ such that $N d(N)$ is even,
construct the \emph{erased random regular} graph on $N$~vertices
as follows:
Initially, attach $d(N)$ \emph{half-edges} to each vertex. 
Call all such half-edges \emph{unpaired}.
At each step, pick one half-edge arbitrarily, and pair it to another
half-edge uniformly at random among all unpaired half-edges
to form an edge, until all the half-edges have been paired.
This results in a uniform random regular multi-graph with degree
$d(N)$~\cite[Proposition 7.7]{remco-book-1}. 
Now the erased random regular graph is formed by erasing all the
self-loops and multiple edges, which then produces a simple graph.

\begin{theorem}
\label{th:reg}
Let $\GG= \{G_N\}_{N\geq 1}$ be a sequence of erased random regular
graphs with degree $d(N)$. Then
{\normalfont (i)}
If $d(N)\to\infty$ as $N\to\infty$, then $\GG$ is $N$-optimal.
{\normalfont (ii)}
If $d(N)/(\sqrt{N}\log(N))\to\infty$ as $N\to\infty$, then $\GG$ is $\sqrt{N}$-optimal.
\end{theorem}
\begin{proof}[Proof of Theorem~\ref{th:reg}]
We will again verify the conditions stated in Theorem~\ref{th:det-seq} for fluid and diffusion level optimality.
For $k\geq 1$, denote $(2k-1)!! = (2k-1)(2k-3)\ldots3.1.$
Fix any $\varepsilon>0$.

(i) For any function $n:\N\to\N$,
\begin{equation}\label{eq:fluid-bound-reg}
\begin{split}
&\Pro{\exists\ V_1, V_2\subseteq V_N:|V_1|\geq\varepsilon N,\ |V_2|\geq n(N), E_N(V_1,V_2) = 0}\\&\\
&= \Pro{\exists\ V_1, V_2\subseteq V_N: |V_1|=\varepsilon N,\ |V_2|= n(N), E_N(V_1,V_2) = 0}\\&\\
&\leq {N \choose \varepsilon N} {N-\varepsilon N \choose n(N)}\frac{(Nd(N)(1-\varepsilon)-1)!!}{(Nd(N)-1)!!}\times\frac{(Nd(N)-n(N)d(N)-1)!!}{(Nd(N)(1-\varepsilon)-n(N)d(N)-1)!!}\\&\\
&\lesssim \frac{1}{[\varepsilon^\varepsilon(1-\varepsilon)^{1-\varepsilon}]^N}\times\frac{\big(\frac{N}{n(N)}\big)^{n(N)}}{\big(1-\frac{n(N)}{N(1-\varepsilon)}\big)^{N(1-\varepsilon)}}\times \exp(-\varepsilon n(N)d(N))\\&\\
&\lesssim \dfrac{\exp(-\varepsilon d(N)n(N))\times \exp(n(N)\log(N))}{\exp(N\log [\varepsilon^\varepsilon(1-\varepsilon)^{1-\varepsilon}])\exp(-n(N))}.
\end{split}
\end{equation}
Choosing $n(N) = N/\sqrt{d(N)}$ say, it can be seen that for any $p(N)$ such that $d(N)\to\infty$ as $N\to\infty$, $n(N)/N\to 0$ and the above probability goes to 0. Therefore for any $\varepsilon, \delta>0$, \eqref{eq:fluid-bound-reg} yields
\begin{align*} 
&\Pro{\dis_1(G_N,\varepsilon)>\delta N}
\leq \Pro{\exists\ U\subseteq V_N:\ |U|\geq \varepsilon N \mbox{ and }\com(U)\geq \delta N}\to 0, \mbox{ as } N\to\infty.
\end{align*}

(ii) Again, as in Part (i), for any function $n:\N\to\N$,
\begin{equation}\label{eq:diff-bound-reg}
\begin{split}
&\Pro{\exists\ V_1, V_2\subseteq V_N:|V_1|\geq\varepsilon \sqrt{N},\ |V_2|\geq n(N), E_N(V_1,V_2) = 0}
\\&\\
&= \Pro{\exists\ V_1, V_2\subseteq V_N: |V_1|=\varepsilon \sqrt{N},\ |V_2|= n(N) E_N(V_1,V_2) = 0}
\\&\\
&\leq {N \choose \varepsilon \sqrt{N}} {N-\varepsilon\sqrt{N} \choose n(N)}\frac{(Nd(N)-\varepsilon\sqrt{N}d(N)-1)!!}{(Nd(N)-1)!!}\times\frac{(Nd(N)-n(N)d(N)-1)!!}{(Nd(N)-\varepsilon\sqrt{N}d(N)-n(N)d(N)-1)!!}
\\&\\
&\lesssim \exp\Big(\frac{\varepsilon\sqrt{N}\log(N)}{2}-\frac{n(N)d(N)}{\sqrt{N}}\Big).
\end{split}
\end{equation}
Now, choosing $n(N) = \sqrt{N}/\sqrt{d(N)/(\sqrt{N}\log(N))}$, it can be seen that as $N\to\infty,$ $n(N)/\sqrt{N}\to 0$ and the above probability converges to 0.
 Therefore for any $\varepsilon, \delta>0$, \eqref{eq:diff-bound-reg} yields
\begin{align*} 
&\Pro{\dis_2(G_N,\varepsilon)>\delta \sqrt{N}}
\leq \Pro{\exists\ U\subseteq V_N:\ |U|\geq \varepsilon \sqrt{N} \mbox{ and }\com(U)\geq \delta \sqrt{N}}\to 0,
\end{align*}
as $N\to\infty$.
\end{proof}

Note that due to Theorem~\ref{th:bdd-deg}, we can conclude that the growth rate condition on degrees for $N$-optimality in Corollary~\ref{cor:errg}~(i) and Theorem~\ref{th:reg}~(i) is not only sufficient, but necessary as well.
Thus informally speaking, $N$-optimality is achieved under the minimum condition required as long as the underlying topology is suitably random.

\begin{figure}
\begin{center}
\includegraphics[width=85mm]{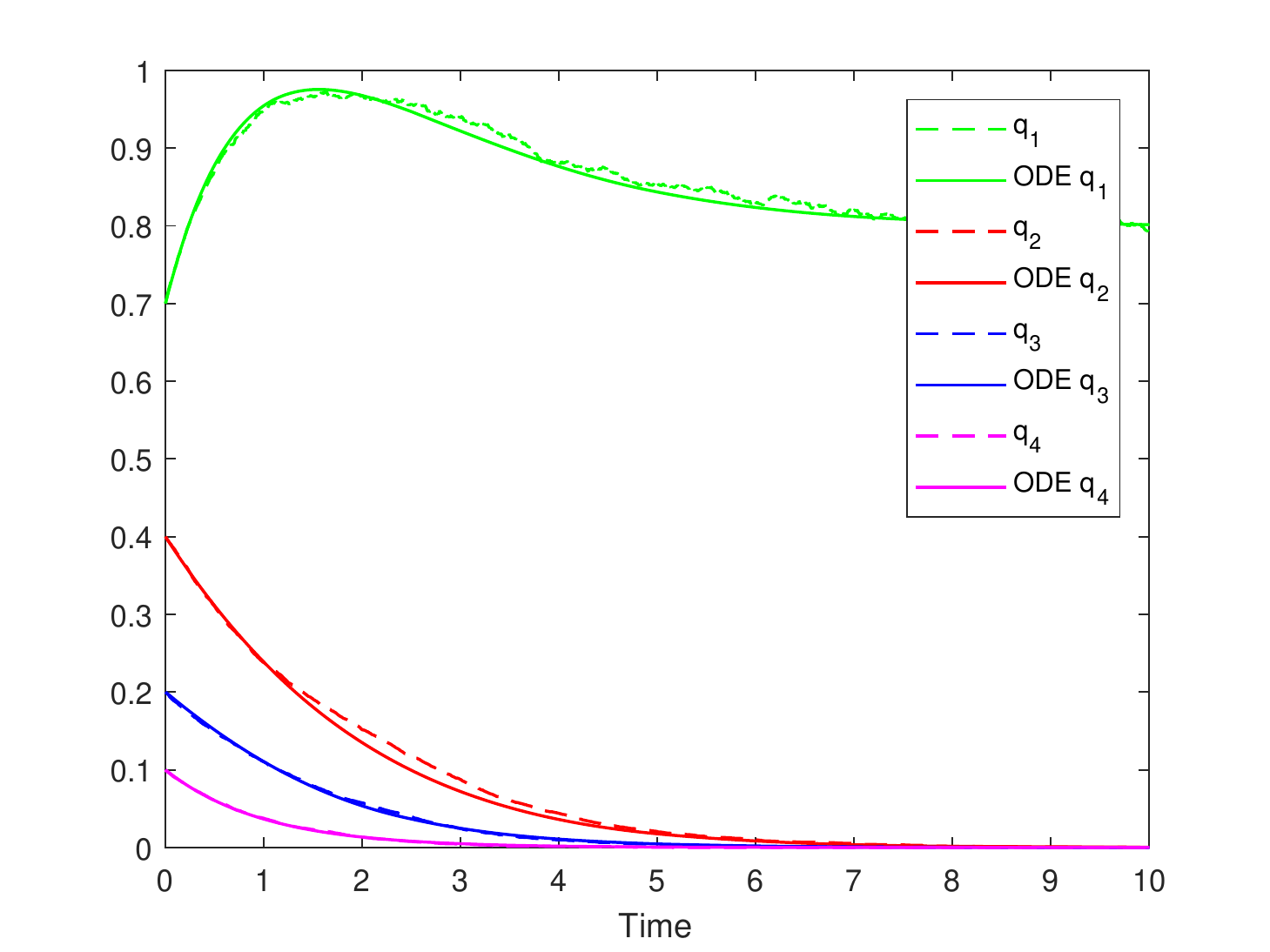}
\end{center}
\caption{Illustration of the fluid-limit trajectories for $\lambda=0.8$ along with a simulation for $N = 10^4$ servers. 
The topology is a single instance of the ERRG on $N=10^4$ vertices with edge probability $1/\sqrt{N}=10^{-2}$, i.e.~the average degree is~100.}
\label{fig:trajectory}
\end{figure}
\begin{figure}
\begin{center}
\includegraphics[width=85mm]{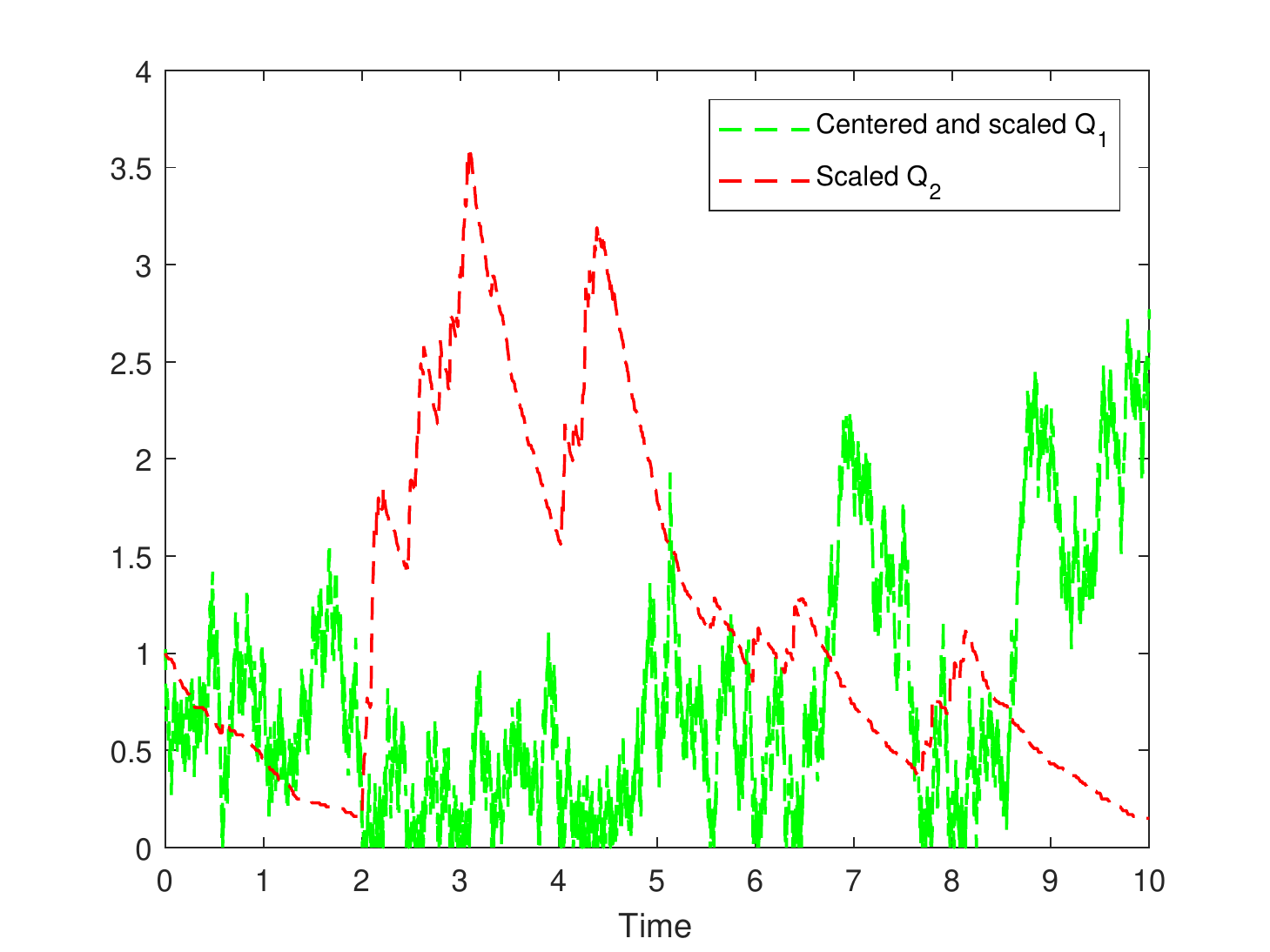}
\end{center}
\caption{Illustration of the diffusion-scaled trajectories in the Halfin-Whitt heavy-traffic regime, for $N = 10^4$ servers and $\lambda(N)= N - \sqrt{N}=9900$.
The topology is a single instance of the ERRG on $N=10^4$ vertices with edge probability $\log(N)^2/\sqrt{N}=0.8483$, i.e.~the average degree is 8483.
}
\label{fig:ht}
\end{figure}
\begin{figure}
\begin{center}
\includegraphics[width=85mm]{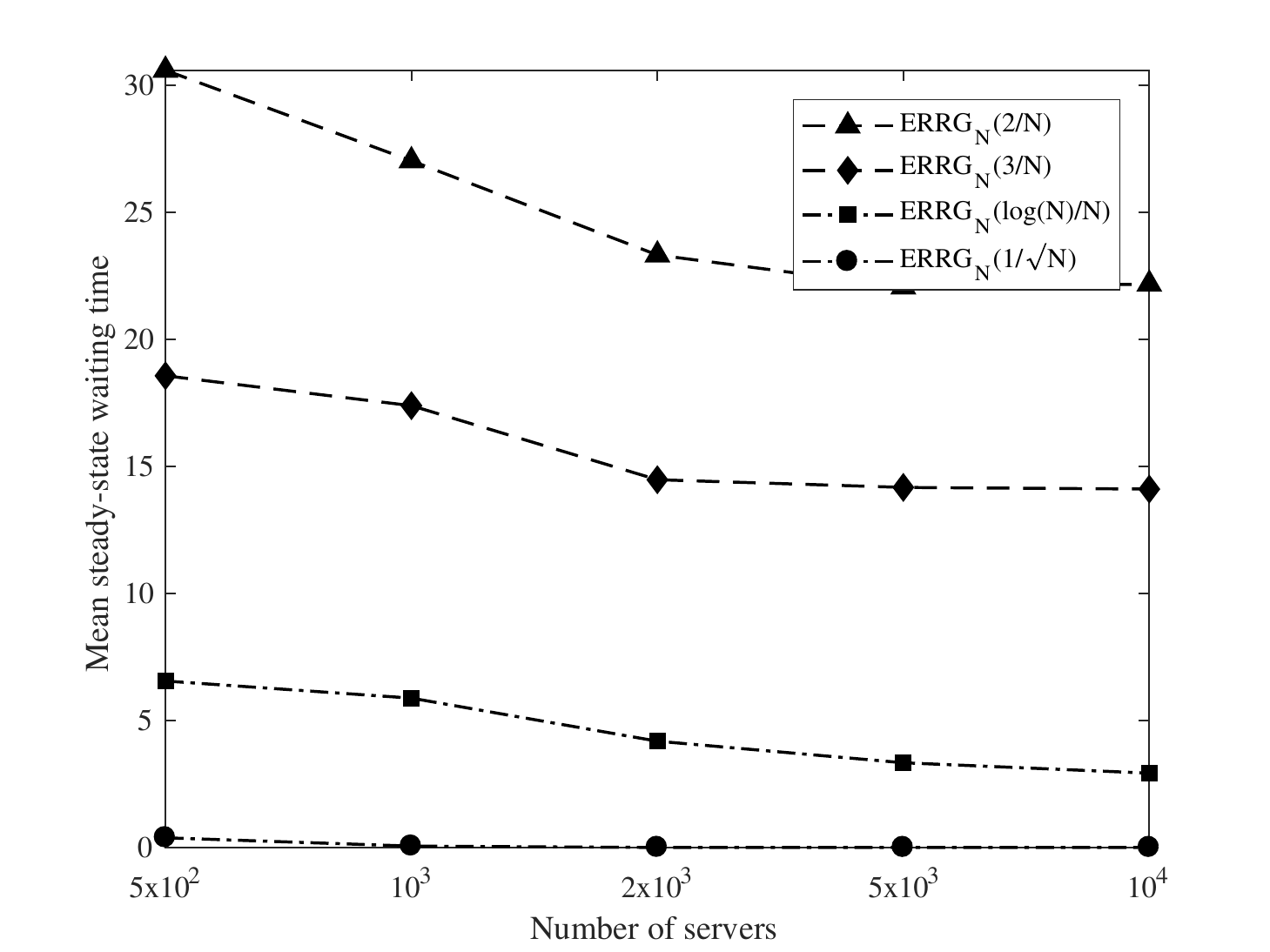}
\end{center}
\caption{Mean steady-state waiting times for $\lambda = 0.9$ and increasing number of servers in ERRG on $N$ vertices with edge probability $c(N)/N$, for $c(N) = 2,3,\log(N),$ and  $\sqrt{N}$.
}
\label{fig:steady-conv}
\end{figure}
\begin{figure}
\begin{center}$
\begin{array}{ccc}
\includegraphics[width=85mm]{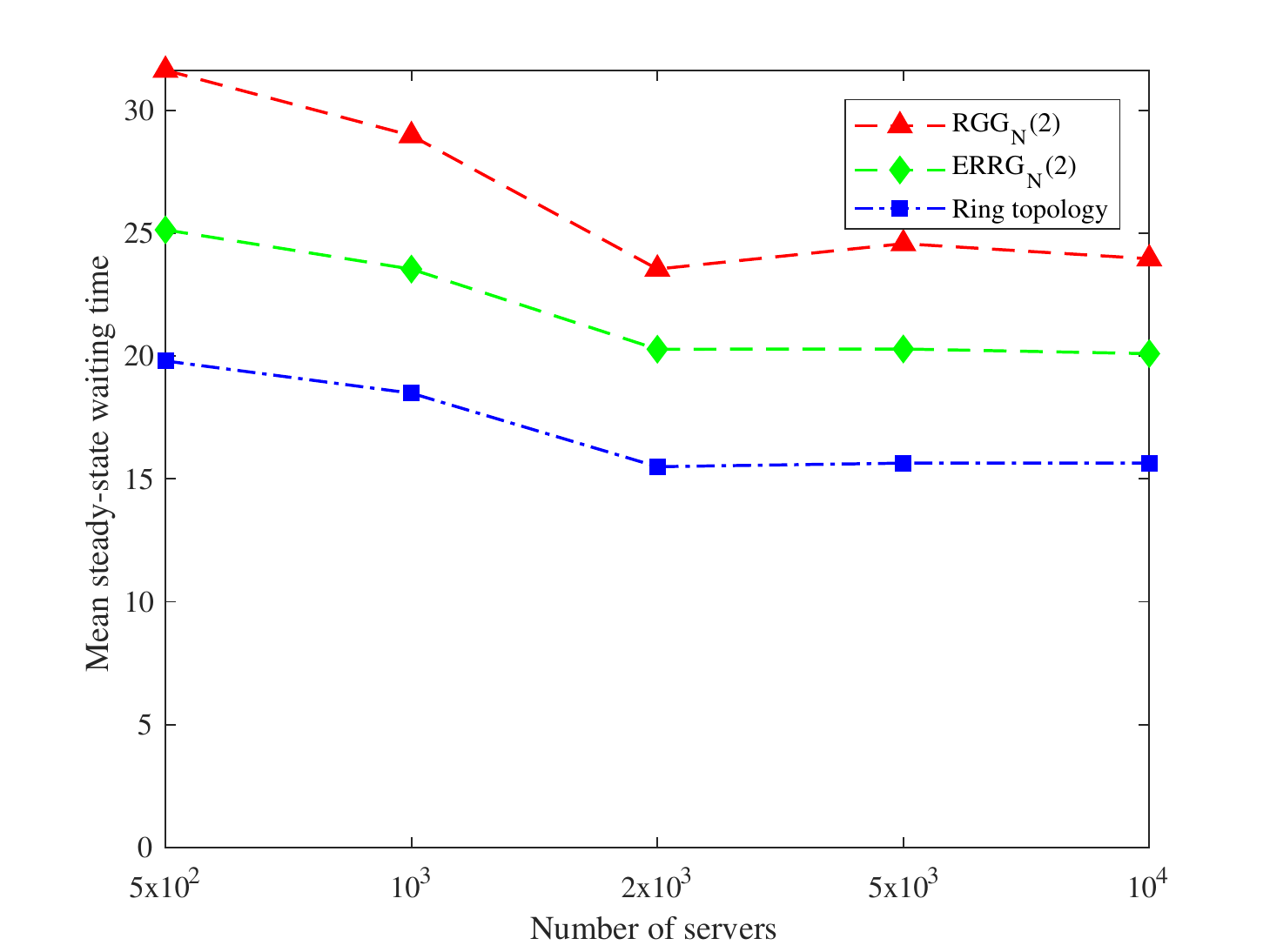}\\
\includegraphics[width=85mm]{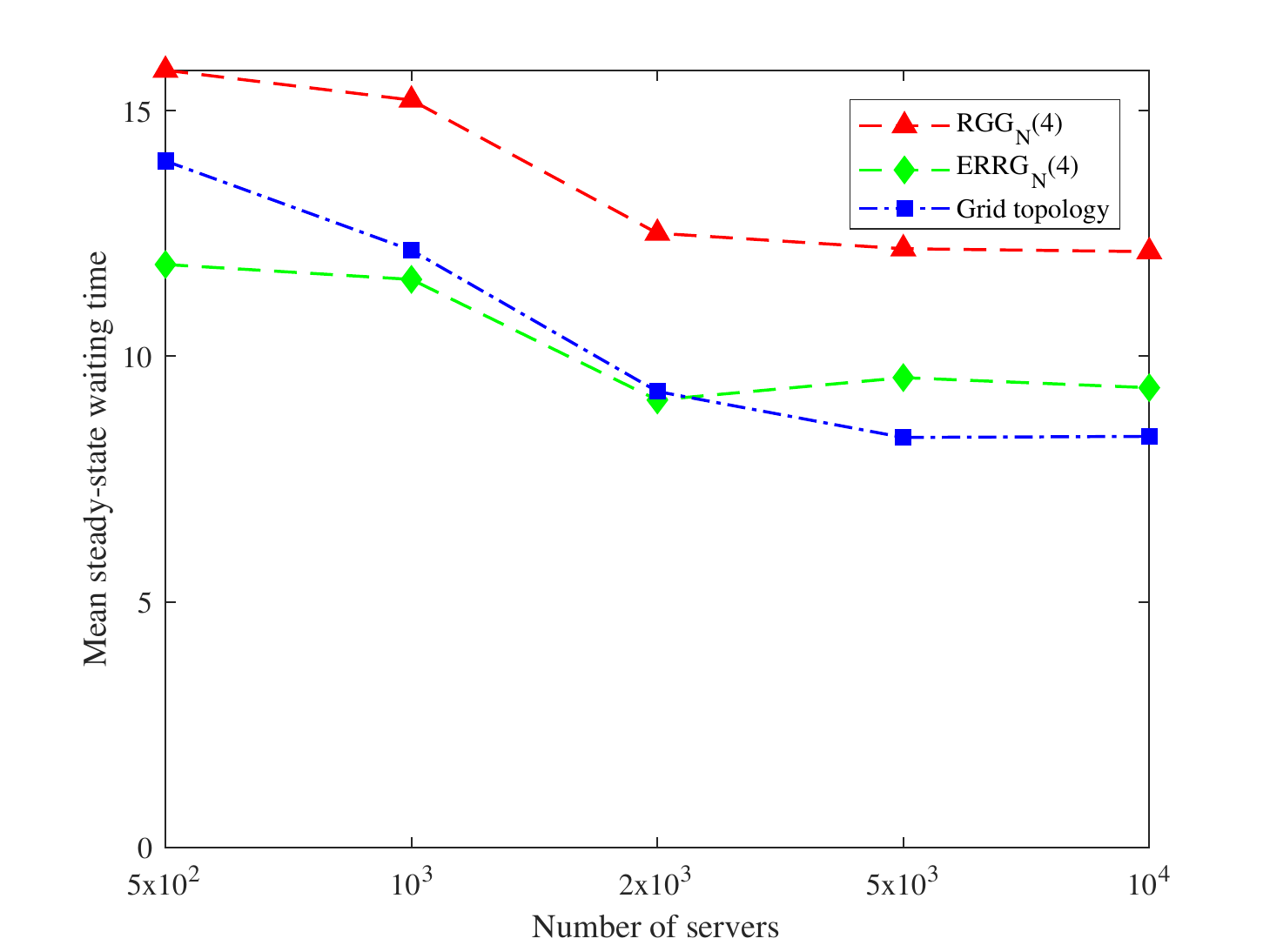}
\end{array}$
\end{center}
\caption{(Top) Performance of the ring topology, and the RGG and ERRG with average degree 2 compared in terms of mean steady-state waiting times. (Bottom) Performance of the grid topology, and the RGG and ERRG with average degree 4 compared in terms of mean steady-state waiting times.}
\label{fig:ring-grid}
\end{figure}
\section{Simulation experiments}
\label{sec:simulation}
In this section we present extensive simulation results to illustrate the fluid and diffusion-limit results, and compare the performance of various graph topologies in terms of mean waiting times.

\vspace{.25cm}
\noindent
{\bf Convergence of sample paths to fluid and diffusion-limit trajectories.}
The fluid-limit trajectory for $\lambda=0.8$ is illustrated in Figure~\ref{fig:trajectory} along with a simulation for $N=10^4$ servers. 
The solid curves represent the case of a clique (i.e.~corresponding to the limit of the occupancy states for the ordinary JSQ policy) as described in Theorem~\ref{fluidjsqd} in the appendix. The dotted lines correspond to the empirical occupancy process when
the underlying graph topology is a single instance of the Erd\H{o}s-R\'enyi random graph (ERRG) on $N=10^4$ vertices with edge probability $1/\sqrt{N}=10^{-2}$, so the average degree is~100.
Even for a topology much sparser than a clique and  finite $N$-value, the simulated path matches closely with the limiting ODE.
In particular, the above suggests that for a large but finite degree, the behavior may be hard to distinguish from the optimal one for all practical purposes, and there seems to be no prominent effect of graph topologies provided the underlying topology is suitably random.

The diffusion-scaled trajectory has been simulated for $N=10^4$ servers in Figure~\ref{fig:ht}. 
The system load $1-1/\sqrt{N}=0.99$ is quite close to 1. 
The underlying graph topology is taken to be a single instance of the ERRG on $N$ vertices with edge probability $\log(N)^2/\sqrt{N}$.
The green and red curves in Figure~\ref{fig:ht} correspond to the centered and scaled occupancy state processes $-\bQ_1(G_N,\cdot)$ and $\bQ_2(G_N,\cdot)$, respectively.
As stated in Corollary~\ref{cor:errg}, the centered and diffusion-scaled trajectories can be observed to be recurrent, and the rate of decrease $\bQ_2(G_N,\cdot)$ seems to be proportional to its value --- resembling some properties of the reflected Ornstein-Uhlenbeck process as in the case of a clique (i.e.~the limit of the ordinary JSQ policy) as stated in Theorem~\ref{diffusionjsqd} in the appendix.

\vspace{.25cm}
\noindent
{\bf Convergence of steady-state waiting times.} 
Figure~\ref{fig:steady-conv} exhibits convergence of mean steady-state waiting times to their limiting values as $N\to\infty$.
By virtue of Little's law, note that the asymptotic mean steady-state waiting time can be expressed in terms of the fixed point of the fluid limit as $\lambda^{-1}\sum_{i\geq 2}q_i$.
For each $N$ and average degree $c(N)$ with $c(N)=2$, 3, $\log(N)$, and $\sqrt{N}$, an instance of ERRG on $N$ vertices with average degree $c(N)$ is taken and the time-averaged value of $\lambda^{-1}\sum_{i\geq 2}q_i^N(t)$ is plotted.
The average is taken over the time interval 0 to 200 or 250 depending on the value of $N$.
The figure shows that if the average degree grows with $N$, then the mean steady-state waiting time converges to zero, while it stays bounded away from zero in case the average degree is constant.
It can further be observed that the convergence is notably fast for a higher growth rate of the average degree.

\vspace{.25cm}
\noindent
{\bf Effect of the topology in sparse case.}
When the average degree is fixed, the effect of the topology seems to be quite prominent. This has also been observed in prior work~\cite{T98, G15}. Specifically, when comparing graphs with average degree 2, it can be seen in the top chart in Figure~\ref{fig:ring-grid} that the ring topology  has a lower mean steady-state waiting time than random topologies (ERRG or RGG). In case of average degree 4, the (toric) grid topology performs worse for small $N$-values, but the performance improves as $N$ increases. 
There are two crucial effects at play here: 
(i) The regularity in degrees of the vertices: Given a mean degree, higher variability (e.g.~presence of many isolated vertices) is expected to degrade the performance and 
(ii) The locality of the connections: Higher diversity in the connections (i.e., graphs with good expander properties) is expected to improve the performance. 
The RGG has a disadvantage in both these aspects: it contains many isolated vertices and also, its connections are highly localized, and thus its performance is consistently worse in both top and bottom charts in Figure~\ref{fig:ring-grid}. 
The ERRG and the lattice graphs (ring/grid) are good with respect to the degree variability and the connection locality, respectively. However, the presence of many isolated vertices hurts more than the benefit provided by the non-local connections when the average degree is small, as exhibited in Figure~\ref{fig:ring-grid}. In case of higher average degree, the number of isolated vertices in the ERRG is relatively small, and thus the benefit from the non-local connections becomes somewhat prominent for smaller $N$-values. It is therefore worthwhile to note that in case of increasing average degrees, the effect of topology becomes less significant, and so the behavior of random topologies (ERRG, RGG, or random regular graphs) turns out to be as good as the clique.

\vspace{.25cm}
\noindent
{\bf Effect of load on the growth rate of the average degree.}
It is expected that if the system is heavily loaded (i.e., $\lambda$ close to~1), then the rate of convergence of the steady-state measure, and hence that of the mean steady-state waiting time becomes slower.
This can be observed in Figure~\ref{fig:lambdaeffect}.
For moderately loaded systems viz.~$\lambda=0.65$ or $0.75$, the convergence is fast even for topologies that are far from fully connected with average degree as low as $\log(N)$. 
\begin{figure}
\begin{center}
\includegraphics[width=85mm]{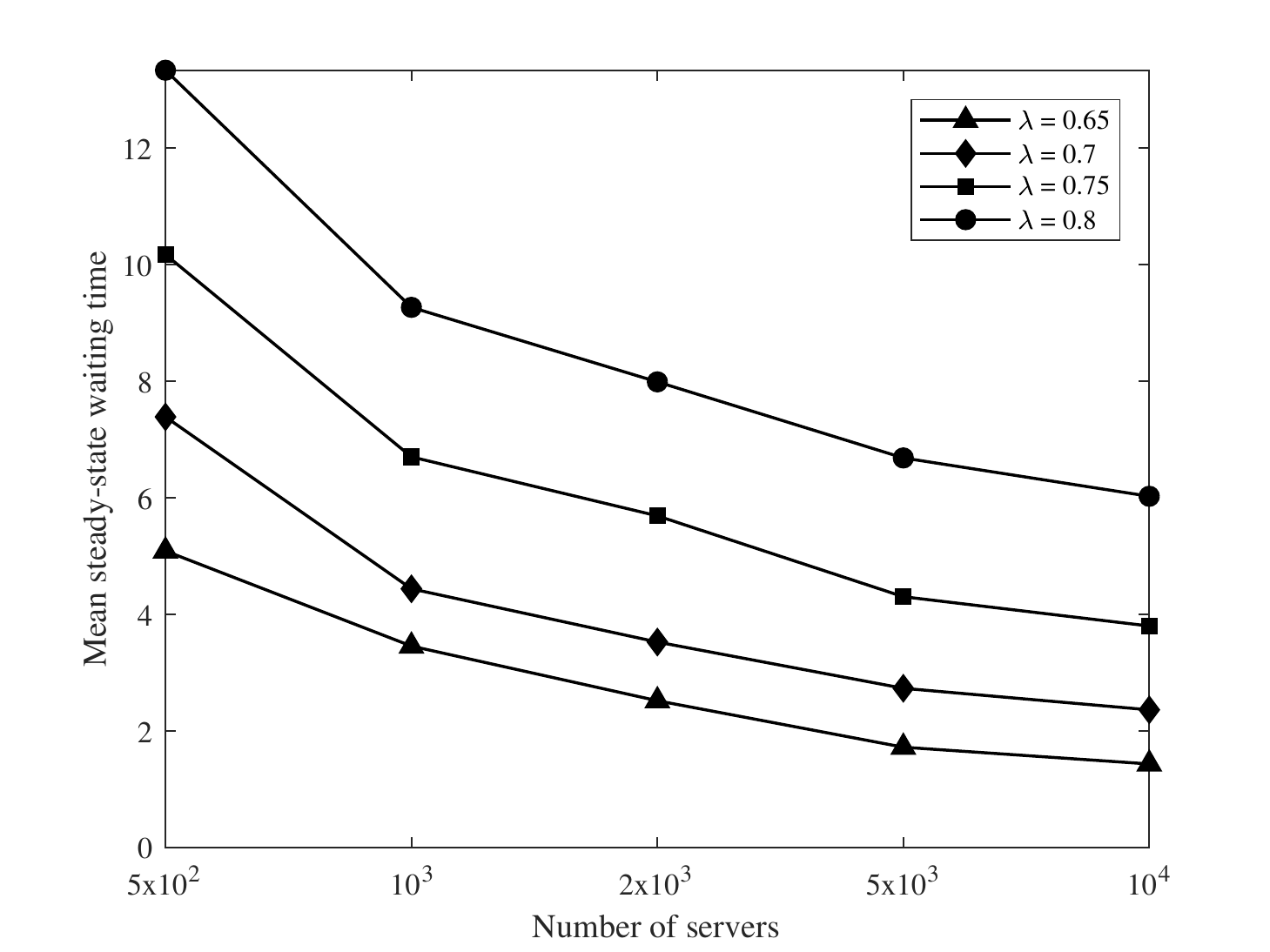}
\end{center}
\caption{Effect of $\lambda$ on the rates of convergence of mean steady-state waiting times.
The underlying topology is an ERRG on $N$ vertices with edge probability $\log(N)/N$, for an increasing number of servers.
}
\label{fig:lambdaeffect}
\end{figure}

\vspace{.25cm}
\noindent
{\bf Performance for spatial random network models.}
The conditions stated in Theorem~\ref{th:det-seq} demand that \emph{any} two \emph{large} portions of the graph share \emph{many} cross edges.
This property is often violated in spatial graph models, where vertices that are closer to each other have a higher tendency to share an edge.
A canonical model for spatial networks is the random geometric graph (RGG), where $N$ vertices correspond to $N$ uniform random locations on $[0,1]^2$ with periodic boundary, and any two vertices share an edge if they are less than a distance $r(N)$ apart.
Note that the average degree in that case is given by $c(N) = (N-1)\pi r(N)^2$.
In other words, for fixed values of $N$ and $c(N)$, the distance $r=r(N)$ scales as $r(N) = \sqrt{c(N)/(\pi N)}$.
To analyze the load balancing process on spatial random graph models, we simulated the processes where the underlying topologies are instances of RGGs on $N$ vertices and average degrees~2,~3, $\log(N)$, and $\sqrt{N}$, and plotted the corresponding mean steady-state waiting times for increasing values of $N$ in Figure~\ref{fig:steady-conv-rgg}.
\begin{figure}
\begin{center}
\includegraphics[width=85mm]{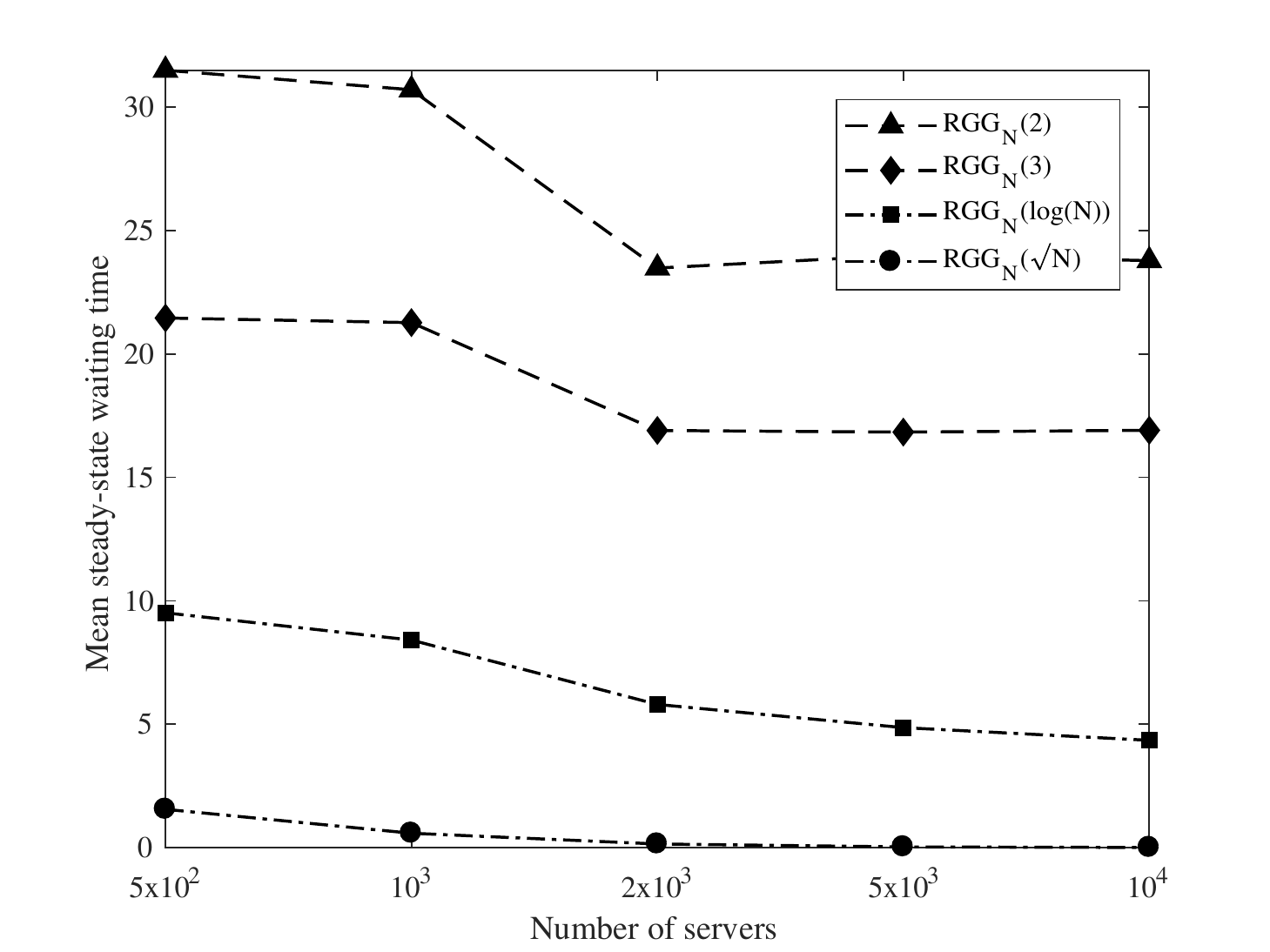}
\end{center}
\caption{Mean steady-state waiting times for $\lambda = 0.8$ and increasing number of servers in RGG on $N$ vertices with average degree $c(N)$, for $c(N) = 2,3,\log(N),$ and  $\sqrt{N}$.
}
\label{fig:steady-conv-rgg}
\end{figure}
The surprising resemblance with the ERRG scenario as depicted in Figure~\ref{fig:steady-conv} hints that the asymptotic optimality result can be preserved even under possibly a relaxed set of conditions.
This motivates future study of the asymptotic optimality 
beyond the classes of graphs we considered.

\section{Conclusion}\label{sec:conclusion}
We have considered load balancing processes in large-scale systems where the servers are inter-connected
by some graph topology.  
For arbitrary topologies we established sufficient criteria for which the performance is asymptotically similar to that in a clique, and hence optimal on suitable scales.  Leveraging these criteria we showed that unlike fixed-degree scenarios (\emph{viz.}~ring, grid) where the topology has a prominent performance impact, the sensitivity to the topology diminishes in the limit when the average degree grows with the number of servers.
In particular, a wide class of suitably random topologies are provably asymptotically optimal.
In other words, the asymptotic optimality of a clique can be achieved while dramatically reducing the number of connections.
In the context of large-scale data centers,
this translates into significant reductions in communication
overhead and storage capacity, since both are roughly proportional
to the number of connections.

Although a growing average degree is necessary in the sense that any graph with finite average degree is sub-optimal,
it is in no way sufficient. 
Load balancing performance can be provably sub-optimal even when the minimum degree is $cN + o(N)$ with $0 < c < 1/2$.  
What happens for $1/2 < c < 1$ is an open question.
Our proof technique relies heavily on a connectivity property entailing that any two sufficiently large
portions of vertices share a lot of edges.  
This property does not hold however in many networks
with connectivity governed by spatial attributes, such as geometric graphs, although the simulation experiments hint that the family of topologies that are asymptotically optimal is likely to be broader than the ERRG and random regular class as considered in the present paper.  
In future research we aim
to examine asymptotic optimality properties of such spatial network models.

\section*{Acknowledgment}
The authors thank Nikhil Bansal for helpful discussions in the early stage of this work, and also for pointing out several relevant references.
The work was financially supported by The Netherlands Organization for Scientific Research (NWO) through Gravitation Networks grant 024.002.003 and TOP-GO grant 613.001.012.

\bibliographystyle{abbrv}

\bibliography{bibl}

\end{document}